\documentclass[preprint,12pt]{elsarticle}

\usepackage{graphicx, color}

\usepackage{amssymb}
\usepackage{amsthm}

\usepackage{amsmath, amsfonts}
\usepackage{url}
\usepackage{subfig}
\usepackage{multirow}

\journal{Journal of Graph Theory}

\graphicspath{
	{./Figures/},
	{./Figures/40/},
	{./Figures/42WA/},
	{./Figures/43/},
	{./Figures/45Grinbergian/},
	{./Figures/45Girth5/},
	{./Figures/48ZZ/},
	{./Figures/57Hatzel/},
	{./Figures/105Thomassen/},
	{./Figures/Transform/},
	{./Figures/Generation/}
}

\newtheorem{theo}{Theorem}[section]

\newtheorem{lem}[theo]{\bf Lemma}
\newtheorem{cor}[theo]{Corollary}

\begingroup\makeatletter\ifx\SetFigFont\undefined%
	\gdef\SetFigFont#1#2#3#4#5{%
		\reset@font\fontsize{#1}{#2pt}%
		\fontfamily{#3}\fontseries{#4}\fontshape{#5}%
	\selectfont}%
\fi\endgroup%
	
\begin{document}

\begin{frontmatter}



\title{Planar Hypohamiltonian Graphs on 40 Vertices}
\author{Mohammadreza Jooyandeh}
\ead{reza@jooyandeh.com}
\ead[url]{http://www.jooyandeh.com}
\author{Brendan D. McKay}
\ead{bdm@cs.anu.edu.au}
\ead[url]{http://cs.anu.edu.au/~bdm}
\address{Research School of Computer Science, Australian National University, ACT 0200, Australia}
\author{Patric R. J. \"{O}sterg{\aa}rd}
\ead{patric.ostergard@aalto.fi}
\author{Ville H. Pettersson}
\ead{ville.pettersson@aalto.fi}
\address{Department of Communications and Networking, Aalto University School of Electrical Engineering, P.O. Box 13000, 00076 Aalto, Finland}
\author{Carol T. Zamfirescu}
\ead{czamfirescu@gmail.com}
\address{Fakult\"{a}t f\"{u}r Mathematik, Technische Universit\"{a}t Dortmund, 44227 Dortmund, Germany}


\begin{abstract}
A graph is hypohamiltonian if it is not Hamiltonian, but the deletion of any single vertex gives a Hamiltonian graph. Until now, the smallest known planar hypohamiltonian graph had 42 vertices, a result due to Araya and Wiener. That result is here improved upon by 25 planar hypohamiltonian graphs of order 40, which are found through computer-aided generation of certain families of planar graphs with girth 4 and a fixed number of \mbox{4-faces.} It is further shown that planar hypohamiltonian graphs exist for all orders greater than or equal to 42. If Hamiltonian cycles are replaced by Hamiltonian paths throughout the definition of hypohamiltonian graphs, we get the definition of hypotraceable graphs. It is shown that there is a planar hypotraceable graph of order 154 and of all orders greater than or equal to~156. We also show that the smallest planar hypohamiltonian graph of girth~5 has 45 vertices.
\end{abstract}


\end{frontmatter}



\section{Introduction}\label{sec:introduction}
\noindent A graph $G = (V,E)$ is called \emph{hypohamiltonian} if it is not Hamiltonian, but the deletion of any single vertex $v \in V$ gives a Hamiltonian graph. The smallest hypohamiltonian graph is the Petersen graph, and all hypohamiltonian graphs with fewer than 18 vertices have been classified \cite{AMW}: there is one such graph for each of the orders 10, 13, and 15, four of order 16, and none of order~17. Moreover, hypohamiltonian graphs exist for all orders greater than or equal to 18.

Chv\'{a}tal \cite{C} asked in 1973 whether there exist \emph{planar} hypohamiltonian graphs, and there was a conjecture that such graphs might not exist \cite{Gru}. However, an infinite family of planar hypohamiltonian graphs was later found by Thomassen \cite{T4}, the smallest among them having order 105. This result was the starting point for work on finding the smallest possible order of such graphs, which has led to the discovery of planar hypohamiltonian graphs of order 57 (Hatzel \cite{H} in 1979), 48 (C.~Zamfirescu and T.~Zamfirescu \cite{ZZ} in 2007), and 42 (Araya and Wiener \cite{WA} in 2011). These four graphs are depicted in Figure~\ref{fig:records}.

\begin{figure}[htbp]
\begin{center}
\begin{minipage}[b]{0.24\linewidth}
\centering
\includegraphics[height=3.2cm]{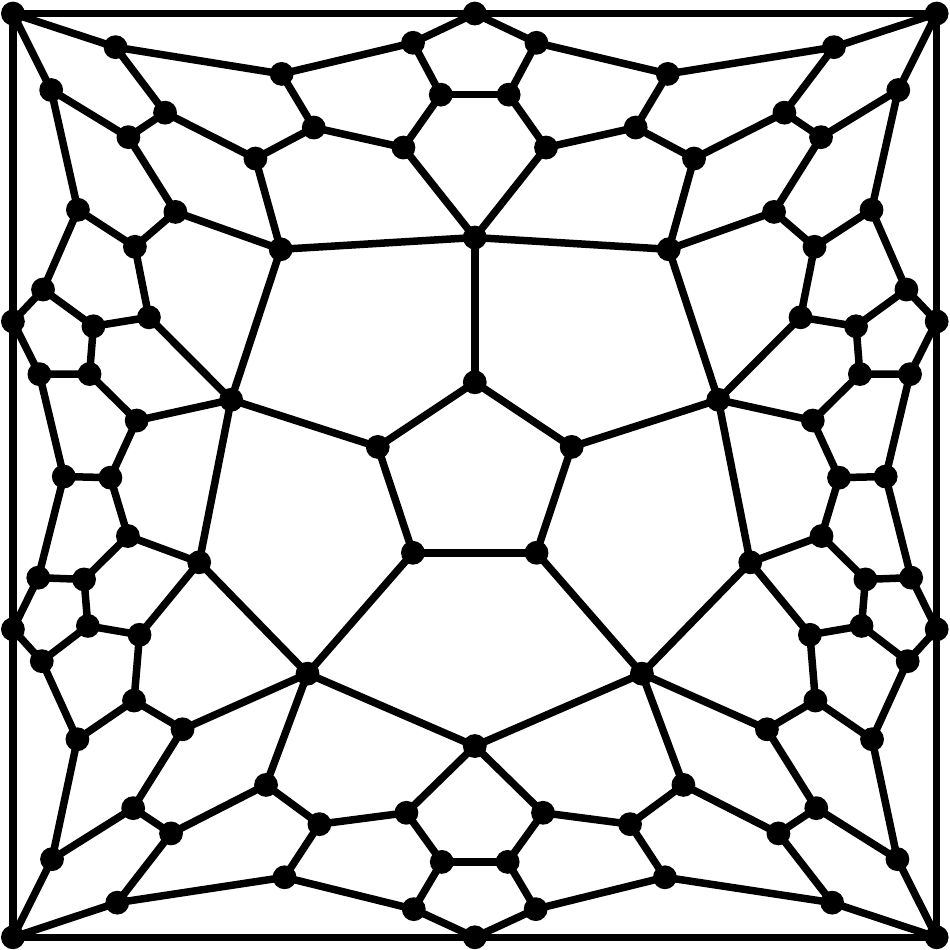}
\end{minipage}
\begin{minipage}[b]{0.24\linewidth}
\centering
\includegraphics[height=3.2cm]{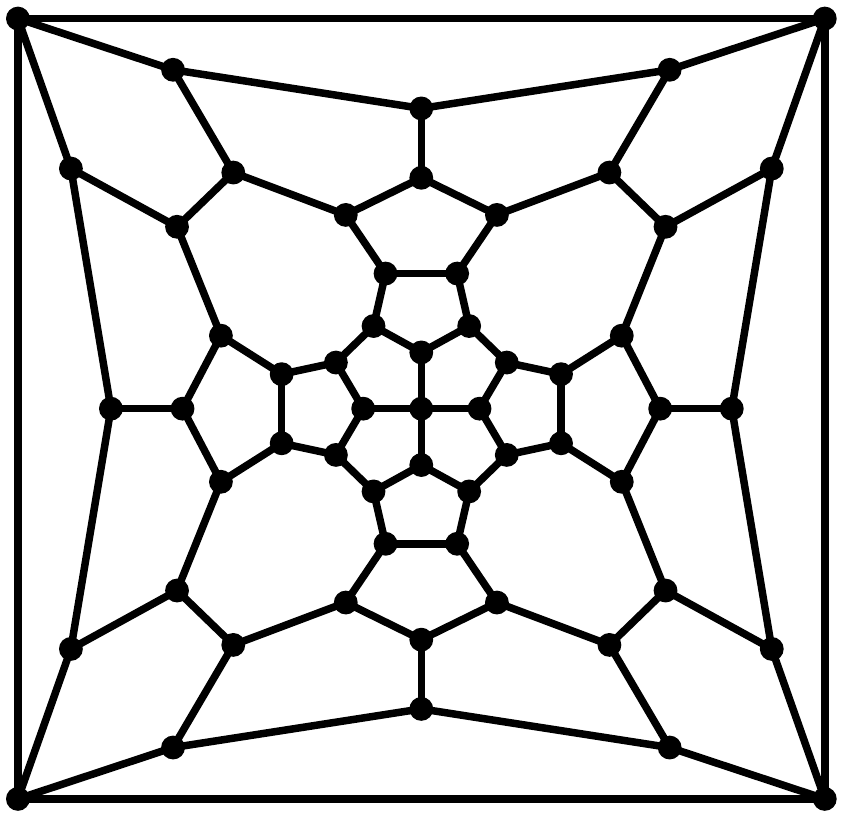}
\end{minipage}
\begin{minipage}[b]{0.24\linewidth}
\centering
\includegraphics[height=3.2cm]{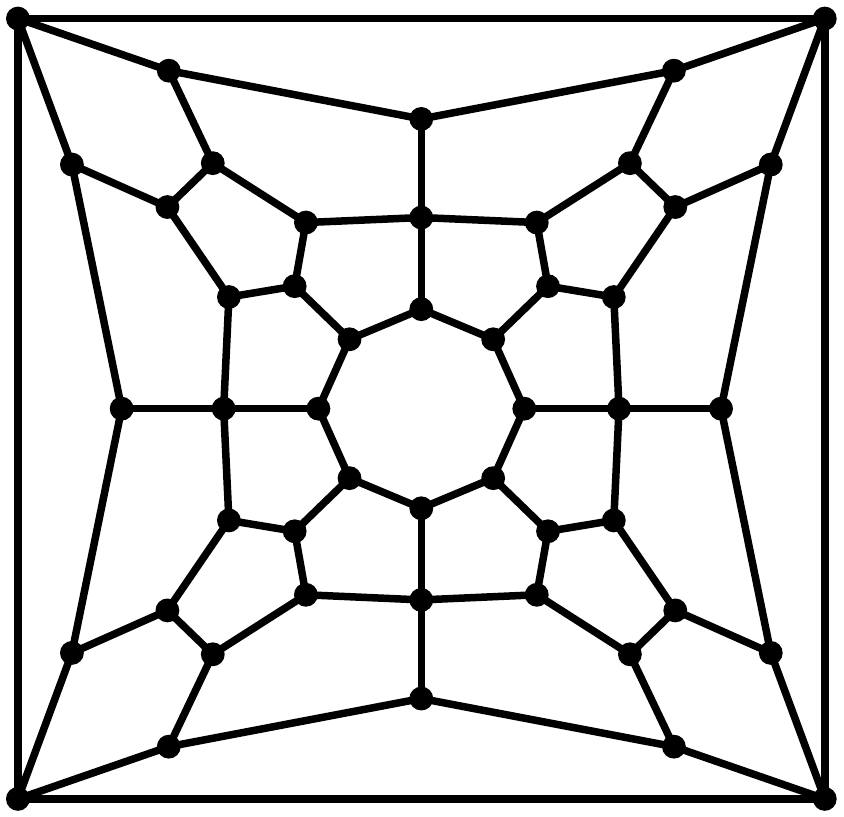}
\end{minipage}
\begin{minipage}[b]{0.24\linewidth}
\centering
\includegraphics[height=3.2cm]{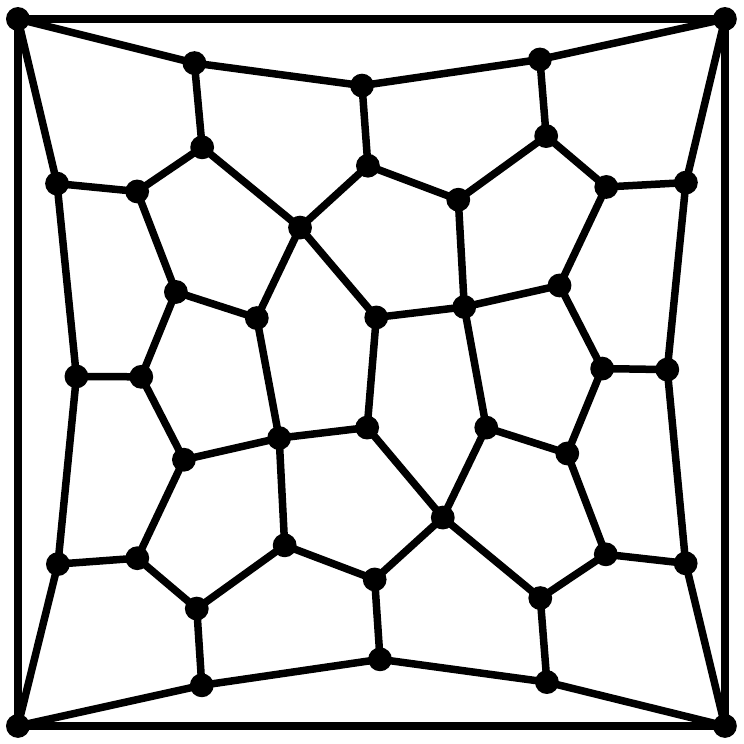}
\end{minipage}
\end{center}
\caption{Planar hypohamiltonian graphs of order 105, 57, 48, and 42}
\label{fig:records}
\end{figure}

Grinberg \cite{G} proved a necessary condition for a plane graph to be Hamiltonian. All graphs in Figure~\ref{fig:records} have the property that one face has size 1 modulo 3, while all other faces have size 2 modulo 3. Graphs with this property are natural candidates for being hypohamiltonian, because they do not satisfy Grinberg's necessary condition for being Hamiltonian. However, we will prove that this approach cannot lead to hypohamiltonian graphs of order smaller than 42. Consequently we seek alternative methods for finding planar hypohamiltonian graphs. In particular, we construct a certain subset of graphs with girth 4 and a fixed number of faces of size 4 in an exhaustive way. This collection of graphs turns out to contain 25 planar hypohamiltonian graphs of order 40.

In addition to finding record-breaking graphs of order 40, we shall prove that planar hypohamiltonian graphs exist for all orders greater than or equal to 42 (it is proved in \cite{WA} that they exist for all orders greater than or equal to~76). Similar results are obtained for \emph{hypotraceable} graphs, which are graphs that do not contain a Hamiltonian path, but the graphs obtained by deleting any single vertex do contain such a path. We show that there is a planar hypotraceable graph of order 154 and of all orders greater than or equal to~156; the old records were 162 and 180, respectively~\cite{WA}.

T.~Zamfirescu defined $\overline{C^i_k}$ and $\overline{P^i_k}$ to be the smallest order for which there is a planar $k$-connected graph such that every set of $i$ vertices is disjoint from some longest cycle and path, respectively \cite{Z2}. Some of the best bounds known so far were $\overline{C^1_3}\leq 42$, $\overline{C^2_3}\leq 3701$, $\overline{P^1_3}\leq 164$, and $\overline{P^2_3}\leq 14694$, which were found based on a planar hypohamiltonian graph on 42 vertices \cite{WA}. We improve upon these bounds using our graphs to $\overline{C^1_3}\leq 40$, $\overline{C^2_3}\leq 2625$, $\overline{P^1_3}\leq 156$, and $\overline{P^2_3}\leq 10350$.

The paper is organized as follows. In Section~\ref{sec:grinbergian} we define Grinbergian graphs and prove theorems regarding their hypohamiltonicity. In Section~\ref{sec:generation} we describe generation of certain planar graphs with girth 4 and a fixed number of faces of size~4, and show a summary of hypohamiltonian graphs found among them. In Section~\ref{sec:corollaries} we present various corollaries based on the new hypohamiltonian graphs. The paper is concluded in Section~\ref{sec:con}.

\section{Grinbergian graphs}\label{sec:grinbergian}
\noindent Consider a plane hypohamiltonian graph $G = (V,E)$.
Since the existence of Hamiltonian cycles is not affected by loops or parallel edges, we can assume that $G$, and all graphs in this paper unless otherwise indicated, are simple.
Let $\kappa(G)$, $\delta(G)$, and $\lambda(G)$ denote the vertex-connectivity, minimum degree, and edge-connectivity of $G$, respectively. We will tacitly use the following fact.

\begin{theo}
$\kappa(G) = \lambda(G) = \delta(G) = 3$.
\end{theo}

\begin{proof}
Since the deletion of any vertex in $V$ gives a Hamiltonian graph, we have $\kappa(G) \geq 3$.
Thomassen \cite{T3} showed that $V$ must contain a vertex of degree 3, so $\delta(G) \leq 3$.
The result now follows from the fact that $\kappa(G)\le\lambda(G)\le\delta(G)$.
\end{proof}

The set of vertices adjacent to a vertex $v$ is denoted by $N(v)$. Let $n = |V|$, $m = |E|$, and $f$ be the number of faces of the plane graph $G$. They satisfy Euler's formula $n-m+f=2$. A $k$-\emph{face} is a face of $G$ which has size~$k$.
For $j=0,1,2$, define ${\cal P}_j={\cal P}_j(G)$ to be the set of
faces of $G$ with size congruent to $j$ modulo~3.

\begin{theo}[{Grinberg's Theorem \cite[Theorem 7.3.5]{W}}]
\label{thm:grinberg}
Given a plane graph with a Hamiltonian cycle\/ $C$ and\/ $f_i$ ($f'_i$)\/ $i$-faces inside (outside) of\/ $C$, we have $$\sum_{i \ge 3} (i-2)(f_i - f'_i) = 0.$$
\end{theo}

We call a graph \emph{Grinbergian} if it is 3-connected, planar and of one of the following two types.

\begin{description}
\item[Type 1] Every face but one belongs to ${\cal P}_2$.
\item[Type 2] Every face has even order, and the graph has odd order.
\end{description}

The motivation behind such a definition is that Grinbergian graphs can easily be proven to be non-Hamiltonian using Grinberg's Theorem. Namely, their face sizes are such that the sum in Grinberg's Theorem cannot possibly be zero. Thus, they are good candidates for hypohamiltonian graphs. (For Type 2 graphs, divide all faces into quadrilaterals. Now the number of edges is even. By Grinberg's Theorem, there must be an even number of quadrilaterals, so the number of faces is even. Since the order is odd, Euler's formula yields a contradiction.)

Our definition of Grinbergian graphs contains two types. One could ask, if there are other types of graphs that can be guaranteed to be non-Hamiltonian with Grinberg's Theorem based on only their sequence of face sizes. The following theorem shows that our definition is complete in this sense.

\begin{theo}
\label{thm:grinbergian_motivation}
Consider a $3$-connected simple planar graph with\/ $n$ vertices ($n \leq 42$) and\/ $F_i$ \mbox{$i$-faces} for each~$i$. Then there are non-negative integers $f_i,f_i'$ ($f_i+f_i'=F_i$) satisfying the equation\/ $\sum_{i \ge 3} (i-2)(f_i - f'_i) = 0$ if and only if the graph is not Grinbergian.
\end{theo}

\begin{proof}
Since the graph is simple and 3-connected, every face must have at least $3$ edges. Applying \cite[Theorem 6.1.23]{W} to the dual of the graph gives $\sum_{i \ge 3} i F_i = 2e \leq 6f-12$, where $f$ is the number of faces. Thus, the average face size is at most $6-(12/f)$. In addition, the size of a face has to be smaller than or equal to the number of vertices in the graph.

Given a sequence of face sizes $F_i$, the problem of finding coefficients $f_i,f_i'$ that satisfy the equation can be reduced to a simple knapsack problem. Namely, note that $$\sum_{i \ge 3} (i-2)(f_i - f'_i) = \sum_{i \ge 3} (i-2)(F_i - 2f'_i) = \sum_{i \ge 3} (i-2)F_i - \sum_{i \ge 3} 2(i-2)f'_i,$$ so solving the equation corresponds to solving an instance of the knapsack problem where we have $F_i$ objects of weight $2(i-2)$, and we must find a subset whose total weight is $\sum_{i \ge 3} (i-2)F_i$. The result can be then verified with an exhaustive computer search over all sequences of face sizes that fulfill the above restrictions.
\end{proof}

Although we have only proved Theorem \ref{thm:grinbergian_motivation}
for $n\le 42$, as that is all we need, we consider it likely to hold
for all~$n$.

By Grinberg's Theorem, Grinbergian graphs are non-Hamiltonian. Notice the difference between our definition and that of Zaks \cite{Z}, who defines \emph{non-Grinbergian} graphs to be graphs with every face in ${\cal P}_2$. We call the faces of a Grinbergian graph not belonging to ${\cal P}_2$ \emph{exceptional}.

\begin{theo}
\label{theo:grinbergian_hypohamiltonian_type}
Every Grinbergian hypohamiltonian graph is of Type 1, its exceptional face belongs to\/ ${\cal P}_1$, and its order is a multiple of\/ $3$.
\end{theo}

\begin{proof}
Let $G$ be a Grinbergian hypohamiltonian graph. There are two possible cases, one for each type of Grinbergian graphs.

{\bf Type 1:}
Let the $j$-face $F$ be the exceptional face of $G$ (so $F \notin {\cal P}_2$), and let $v$ be a vertex of $F$. Vertex $v$ belongs to $F$ and to several, say $h$, faces in ${\cal P}_2$. The face of $G - v$ containing $v$ in its interior has size $3h + j - 2\pmod 3$, while all other faces have size $2\pmod 3$. Since $G$ is hypohamiltonian, $G - v$ must be Hamiltonian. Thus, $G - v$ cannot be a Grinbergian graph, so $3h + j - 2 \equiv 2\pmod 3$, thus $F \in {\cal P}_1$.

{\bf Type 2:}  As $G$ contains only cycles of even length, it is bipartite. A bipartite graph can only be Hamiltonian if both of the parts have equally many vertices. Thus, it is not possible that $G-v$ is Hamiltonian for every vertex $v$, so $G$ cannot be hypohamiltonian and we have a contradiction.

Hence, $G$ is of Type~1, and its exceptional face is in ${\cal P}_1$. Counting the edges we get $2m \equiv 2(f-1)+1 \pmod 3$, which together with Euler's formula gives $$2n = 2m-2f+4 \equiv 2f-1-2f+4 \equiv 0 \!\!\!\pmod 3,$$ so $n$ is a multiple of 3.
\end{proof}

\begin{lem}\label{lem:grinberg_type_1_degree_4}
In a Grinbergian hypohamiltonian graph\/ $G$ of Type 1, all vertices of the exceptional face have degree at least\/ $4$.
\end{lem}

\begin{proof}
%
Denote the exceptional face by $Q$. Now assume that there is a vertex $v \in V(Q)$ with degree 3, and consider the vertex $w \in N(v) \setminus V(Q)$. (Note that $N(v) \setminus V(Q) \ne \emptyset$, because $G$ is $3$-connected.) Let $k$ be the degree of~$w$, and denote by $N_1,\ldots,N_k$ the sizes of the faces of $G$ that contain $w$. We have $N_i \equiv 2 \pmod 3$ for all~$i$.
Now consider the graph $G'$ obtained by deleting $w$ from $G$.
The size of the face of $G'$ which in $G$ contained $w$ in its interior is $m = \sum_i (N_i-2) \equiv 0 \pmod 3$. Assume that $G'$ is Hamiltonian. The graph $G'$ contains only faces in ${\cal P}_2$ except for one face in ${\cal P}_1$ and one in ${\cal P}_0$. The face in ${\cal P}_1$ and the face in ${\cal P}_0$ are on different sides of any Hamiltonian cycle in $G'$, since the cycle must pass through~$v$. The sum in Grinberg's Theorem, modulo 3, is then $(m - 2) + 1 \equiv 2 \pmod 3$ or $-(m - 2) - 1 \equiv 1 \pmod 3$, so $G'$ is non-Hamiltonian and we have a contradiction.
\end{proof}

In Section \ref{sec:generation}, we will use these properties to show that the smallest Grinbergian hypohamiltonian graph has 42 vertices.

\section{Generation of 4-face deflatable hypohamiltonian graphs}\label{sec:generation}
\noindent We define the operation \emph{4-face deflater} denoted by $\mathcal{FD}_4$ which squeezes a 4-face of a plane graph into a path of length 2 (see Figure~\ref{fig:operation}). The inverse of this operation is called \emph{2-path inflater} which expands a path of length 2 into a 4-face and is denoted by $\mathcal{PI}_2$. In Figure~\ref{fig:operation} each half line connected to a vertex means that there is an edge incident to the vertex at that position and a small triangle allows zero or more incident edges at that position. For example $v_3$ has degree at least 3 and 4 in Figures~\ref{fig:operation_result} and \ref{fig:operation_original}, respectively. The set of all graphs obtained by applying $\mathcal{PI}_2$ and $\mathcal{FD}_4$ on a graph $G$ is denoted by $\mathcal{PI}_2(G)$ and $\mathcal{FD}_4(G)$, respectively.

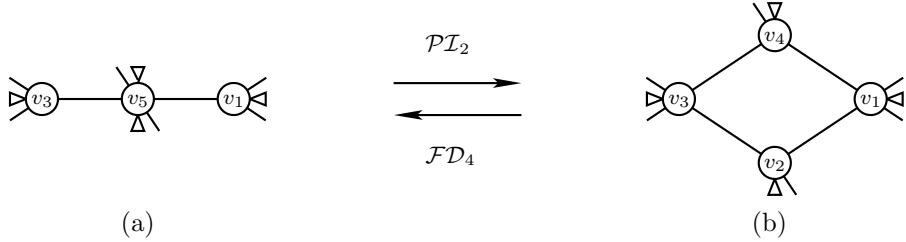
\begin{figure}[htbp]
\begin{center}
\begin{minipage}[b]{0.3\linewidth}
\centering
\subfloat[]{\label{fig:operation_result}
	\begin{picture}(0,0)%
	\includegraphics{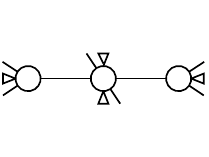}%
	\end{picture}%
	\setlength{\unitlength}{3522sp}%
	\begin{picture}(1852,1394)(-1825,-1883)
	\put(-972,-1225){\makebox(0,0)[lb]{\smash{{\SetFigFont{8}{9.6}{\rmdefault}{\mddefault}{\updefault}{\color[rgb]{0,0,0}$v_5$}%
	}}}}
	\put(-1647,-1225){\makebox(0,0)[lb]{\smash{{\SetFigFont{8}{9.6}{\rmdefault}{\mddefault}{\updefault}{\color[rgb]{0,0,0}$v_3$}%
	}}}}
	\put(-297,-1225){\makebox(0,0)[lb]{\smash{{\SetFigFont{8}{9.6}{\rmdefault}{\mddefault}{\updefault}{\color[rgb]{0,0,0}$v_1$}%
	}}}}
	\end{picture}
}
\end{minipage}
\begin{minipage}[b]{0.3\linewidth}
\centering
	\begin{picture}(0,0)%
	\includegraphics{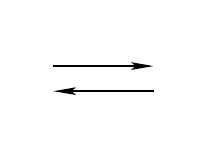}%
	\end{picture}%
	\setlength{\unitlength}{3522sp}%
	\begin{picture}(1852,1394)(-1825,-1883)
	\put(-1124,-849){\makebox(0,0)[lb]{\smash{{\SetFigFont{9}{10.8}{\rmdefault}{\mddefault}{\updefault}{\color[rgb]{0,0,0}$\mathcal{PI}_2$}%
	}}}}
	\put(-1124,-1636){\makebox(0,0)[lb]{\smash{{\SetFigFont{9}{10.8}{\rmdefault}{\mddefault}{\updefault}{\color[rgb]{0,0,0}$\mathcal{FD}_4$}%
	}}}}
	\end{picture}%

\end{minipage}
\begin{minipage}[b]{0.3\linewidth}
\centering
\end{minipage}
\subfloat[]{\label{fig:operation_original}
	\begin{picture}(0,0)%
	\includegraphics{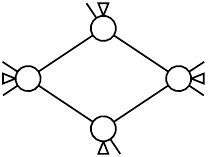}%
	\end{picture}%
	\setlength{\unitlength}{3522sp}%
	\begin{picture}(1852,1394)(-1825,-1883)
	\put(-972,-775){\makebox(0,0)[lb]{\smash{{\SetFigFont{8}{9.6}{\rmdefault}{\mddefault}{\updefault}{\color[rgb]{0,0,0}$v_4$}%
	}}}}
	\put(-972,-1675){\makebox(0,0)[lb]{\smash{{\SetFigFont{8}{9.6}{\rmdefault}{\mddefault}{\updefault}{\color[rgb]{0,0,0}$v_2$}%
	}}}}
	\put(-1647,-1225){\makebox(0,0)[lb]{\smash{{\SetFigFont{8}{9.6}{\rmdefault}{\mddefault}{\updefault}{\color[rgb]{0,0,0}$v_3$}%
	}}}}
	\put(-297,-1225){\makebox(0,0)[lb]{\smash{{\SetFigFont{8}{9.6}{\rmdefault}{\mddefault}{\updefault}{\color[rgb]{0,0,0}$v_1$}%
	}}}}
	\end{picture}%
}
\end{center}
\caption{Operations $\mathcal{FD}_4$ and $\mathcal{PI}_2$}
\label{fig:operation}
\end{figure}

Let $\mathcal{D}_5(f)$ be the set of all simple connected plane graphs with $f$ faces and minimum degree at least 5.
This class of graphs can be generated using the program \emph{plantri} \cite{BM07}. Let us denote the dual of a plane graph $G$ by $G^*$. We define the family of \emph{4-face deflatable graphs} (not necessarily simple) with $f$ 4-faces and $n$ vertices, denoted by $\mathcal{M}^4_f(n)$, recursively as:
\begin{equation}
\mathcal{M}^4_f(n)=\left\{
\begin{array}{lr}
\left\{G^* : G\in\mathcal{D}_5\left(n\right)\right\},&f=0;\\[1ex]
\bigcup_{G\in\mathcal{M}^4_{f-1}(n-1)}\mathcal{PI}_2(G),&f>0.
\end{array}
\right.
\end{equation}
It should be noted that applying $\mathcal{PI}_2$ to a graph increases the number of both vertices and 4-faces by one. Then, we can filter $\mathcal{M}^4_f$ for possible hypohamiltonian graphs and we define $\mathcal{H}^4_f$ based on it as:
\begin{equation}
\mathcal{H}^4_f(n)=\{G\in\mathcal{M}^4_f(n) : \text{$G$ is hypohamiltonian}\}.
\end{equation}

The function $\mathcal{H}^4_f(n)$ can be defined for $n\geq 20$ because the minimum face count for a simple planar 5-regular graph is 20 (icosahedron). Also it is straightforward to check that $f\leq n-20$ because $\mathcal{H}^4_f(n)$ is defined based on $\mathcal{H}^4_{f-1}(n-1)$ for $f>0$.

To test hamiltonicity of graphs, we use depth-first search with the following pruning rule: If there is a vertex that does not belong to the current partial cycle, and has fewer than two neighbours that either do not belong to the current partial cycle or are an endpoint of the partial cycle, the search can be pruned. This approach can be implemented efficiently with careful bookkeeping of the number of neighbours that do not belong to the current partial cycle for each vertex. It turns out to be reasonably fast for small planar graphs.

Finally, we define the set of \emph{4-face deflatable hypohamiltonian graphs} denoted by $\mathcal{H}^4(n)$ as:
\begin{equation}
\mathcal{H}^4(n)=\bigcup_{f=0}^{n-20} \mathcal{H}^4_f(n).
\end{equation}

Using this definition for $\mathcal{H}^4_f(n)$, we are able to find many hypohamiltonian graphs which were not discovered so far. The graphs found on 105 vertices by Thomassen \cite{T4}, 57 by Hatzel \cite{H}, 48 by C.~Zamfirescu and T.~Zamfirescu \cite{ZZ}, and 42 by Araya and Wiener \cite{WA} are all 4-face deflatable and belong to $\mathcal{H}^4_0(105)$, $\mathcal{H}^4_1(57)$, $\mathcal{H}^4_1(48)$, and $\mathcal{H}^4_1(42)$, respectively.

We have generated $\mathcal{H}^4_f(n)$ exhaustively for $20\leq n\leq 39$ and all possible $f$ but no graph was found, which means that for all $n<40$ we have $\mathcal{H}^4_f(n)=\emptyset$. For $n>39$ we were not able to finish the computation for all $f$ due to the amount of required time. For $n=40,41,42,43$ we finished the computation up to $f=12,12,11,10$, respectively. The only values of $n$ and $f$ for which $\mathcal{H}^4_f(n)$ was non-empty were $\mathcal{H}^4_5(40)$, $\mathcal{H}^4_1(42)$, $\mathcal{H}^4_7(42)$, $\mathcal{H}^4_4(43)$, and $\mathcal{H}^4_5(43)$. More details about these families are provided in Tables~\ref{tbl:H_40}, \ref{tbl:H_42} and \ref{tbl:H_43}. Based on the computations we can obtain the Theorems~\ref{theo:no_hypo_less_than_40}, \ref{theo:25_hypo_for_40}, \ref{theo:179_hypo_for_42}, and \ref{theo:497_hypo_for_43}. The complete list of graphs generated is available to download at \cite{J1}.

\begin{table}[htbp]
\centering
\begin{tabular}{|c|l|l|r|}
\hline
4-Face Count&Face Sequence&Degree Sequence&Count\\
\hline
\multirow{5}{*}{5}&\multirow{5}{*}{$5\times 4, 22\times 5$}&$30\times 3, 10\times 4$&4\\
\cline{3-4}
&&$31\times 3, 8\times 4, 1\times 5$&10\\
\cline{3-4}
&&$32\times 3, 6\times 4, 2\times 5$&9\\
\cline{3-4}
&&$33\times 3, 4\times 4, 3\times 5$&2\\
\cline{3-4}
&&All&25\\
\hline
\end{tabular}
\caption{Facts about $\mathcal{H}^4_5(40)$}
\label{tbl:H_40}
\end{table}

\begin{table}[htbp]
\centering
\begin{tabular}{|c|l|l|r|}
\hline
4-Face Count&Face Sequence&Degree Sequence&Count\\
\hline
\multirow{2}{*}{1}&\multirow{2}{*}{$1\times 4, 26\times 5$}&$34\times 3, 8\times 4$&5\\
\cline{3-4}
&&$35\times 3, 6\times 4, 1\times 5$&2\\
\hline
\multirow{9}{*}{7}&\multirow{9}{*}{$7\times 4, 22\times 5$}&$30\times 3, 12\times 4$&4\\
\cline{3-4}
&&$31\times 3, 10\times 4, 1\times 5$&28\\
\cline{3-4}
&&$32\times 3, 8\times 4, 2\times 5$&57\\
\cline{3-4}
&&$33\times 3, 6\times 4, 3\times 5$&49\\
\cline{3-4}
&&$33\times 3, 7\times 4, 1\times 5, \times 6$&11\\
\cline{3-4}
&&$34\times 3, 4\times 4, 4\times 5$&10\\
\cline{3-4}
&&$34\times 3, 5\times 4, 2\times 5, 1\times 6$&5\\
\cline{3-4}
&&$34\times 3, 6\times 4, 2\times 6$&6\\
\cline{3-4}
&&$35\times 3, 4\times 4, 1\times 5, 2\times 6$&2\\
\hline
All&All&All&179\\
\hline
\end{tabular}
\caption{Facts about $\mathcal{H}^4_1(42)$ and $\mathcal{H}^4_7(42)$}
\label{tbl:H_42}
\end{table}

\begin{table}[htbp]
\centering
\begin{tabular}{|c|l|l|r|}
\hline
4-Face Count&Face Sequence&Degree Sequence&Count\\
\hline
\multirow{2}{*}{4}&\multirow{2}{*}{$4\times 4, 23\times 5, 1\times 7$}&$36\times 3, 6\times 4, 1\times 6$&1\\
\cline{3-4}
&&$37\times 3, 4\times 4, 1\times 5, 1\times 6$&1\\
\hline
\multirow{15}{*}{5}&\multirow{5}{*}{$5\times 4, 22\times 5, 1\times 8$}&$34\times 3, 9\times 4$&8\\
\cline{3-4}
&&$35\times 3, 7\times 4, 1\times 5$&20\\
\cline{3-4}
&&$36\times 3, 5\times 4, 2\times 5$&19\\
\cline{3-4}
&&$37\times 3, 3\times 4, 3\times 5$&1\\
\cline{3-4}
&&$37\times 3, 4\times 4, 1\times 5, 1\times 6$&1\\
\cline{2-4}
&\multirow{10}{*}{$5\times 4, 24\times 5$}&$32\times 3, 11\times 4$&52\\
\cline{3-4}
&&$33\times 3, 9\times 4, 1\times 5$&148\\
\cline{3-4}
&&$34\times 3, 7\times 4, 2\times 5$&175\\
\cline{3-4}
&&$34\times 3, 8\times 4, 1\times 6$&2\\
\cline{3-4}
&&$35\times 3, 5\times 4, 3\times 5$&56\\
\cline{3-4}
&&$35\times 3, 6\times 4, 1\times 5, 1\times 6$&6\\
\cline{3-4}
&&$36\times 3, 3\times 4, 4\times 5$&1\\
\cline{3-4}
&&$36\times 3, 4\times 4, 2\times 5, 1\times 6$&4\\
\cline{3-4}
&&$37\times 3, 2\times 4, 3\times 5, 1\times 6$&1\\
\cline{3-4}
&&$37\times 3, 3\times 4, 1\times 5, 2\times 6$&1\\
\hline
All&All&All&497\\
\hline
\end{tabular}
\caption{Facts about $\mathcal{H}^4_4(43)$ and $\mathcal{H}^4_5(43)$}
\label{tbl:H_43}
\end{table}

\begin{theo}\label{theo:no_hypo_less_than_40}
There is no planar\/ $4$-face deflatable hypohamiltonian graph of order less than\/ $40$.
\end{theo}

\begin{theo}\label{theo:25_hypo_for_40}
There are at least\/ $25$ planar\/ $4$-face deflatable hypohamiltonian graphs on\/ $40$ vertices.
\end{theo}

\begin{theo}\label{theo:179_hypo_for_42}
There are at least\/ $179$ planar\/ $4$-face deflatable hypohamiltonian graphs on\/ $42$ vertices.
\end{theo}

\begin{theo}\label{theo:497_hypo_for_43}
There are at least\/ $497$ planar\/ $4$-face deflatable hypohamiltonian graphs on\/ $43$ vertices.
\end{theo}

\begin{lem}\label{lem:fs_simple_dual}
Let\/ $G$ be a Type 1 Grinbergian hypohamiltonian graph whose faces are at least $5$-faces except one which is a\/ $4$-face. Then any\/ $G'$ in\/ $\mathcal{FD}_4(G)$ has a simple dual.

\begin{proof}
As $G$ is a simple $3$-connected graph, the dual $G^*$ of $G$ is simple, too. Let $G'\in\mathcal{FD}_4(G)$ and assume to the contrary that $G'^*$ is not simple.

If $G'^*$ has some multiedges, then the fact that $G^*$ is simple shows that either the two faces incident with $v_1v_5$ or with $v_3v_5$ in Figure~\ref{fig:operation_2_cut_result} (we assume the first by symmetry) have a common edge $v_8v_9$ in addition to $v_1v_5$. Let $v_1v_6$ and $v_1v_7$ be the edges adjacent to $v_1v_5$ in the cyclic order of $v_1$. Note that $v_6\neq v_7$ because $d(G';v_1)\geq 3$ by Lemma~\ref{lem:grinberg_type_1_degree_4}. If $v_1$ and $v_8$ were the same vertex, then $v_1$ would be a cut vertex in $G$ considering the closed walk $v_1v_6\cdots v_8(=\!\!v_1)$. But this is impossible as $G$ is $3$-connected, so $v_1\neq v_8$. Now we can see that $\{v_1,v_8\}$ is a $2$-cut for $G$ considering the closed walk $v_1v_6\cdots v_8\cdots v_7v_1$.

Also, if $G'^*$ has a loop, with the same discussion, we can assume that the two faces incident with $v_1v_5$ are the same but then $v_1$ would be a cut vertex for $G$. Therefore, both having multiedges or having loops violate the fact that $G$ is $3$-connected. So the assumption that $G'^*$ is not simple is incorrect, which completes the proof.
\begin{figure}[htbp]
\begin{center}
\begin{minipage}[b]{0.39\textwidth}
\centering
\subfloat[$G$]{\label{fig:operation_2_cut_original}
\resizebox{0.85\width}{0.85\height}{
\begin{picture}(0,0)%
\includegraphics{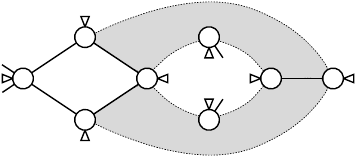}%
\end{picture}%
\setlength{\unitlength}{2901sp}%
\begin{picture}(3873,1692)(-1825,-2032)
\put(378,-775){\makebox(0,0)[lb]{\smash{{\SetFigFont{6}{7.2}{\rmdefault}{\mddefault}{\updefault}{\color[rgb]{0,0,0}$v_6$}%
}}}}
\put(378,-1675){\makebox(0,0)[lb]{\smash{{\SetFigFont{6}{7.2}{\rmdefault}{\mddefault}{\updefault}{\color[rgb]{0,0,0}$v_7$}%
}}}}
\put(1053,-1225){\makebox(0,0)[lb]{\smash{{\SetFigFont{6}{7.2}{\rmdefault}{\mddefault}{\updefault}{\color[rgb]{0,0,0}$v_8$}%
}}}}
\put(1728,-1225){\makebox(0,0)[lb]{\smash{{\SetFigFont{6}{7.2}{\rmdefault}{\mddefault}{\updefault}{\color[rgb]{0,0,0}$v_9$}%
}}}}
\put(-972,-775){\makebox(0,0)[lb]{\smash{{\SetFigFont{6}{7.2}{\rmdefault}{\mddefault}{\updefault}{\color[rgb]{0,0,0}$v_4$}%
}}}}
\put(-972,-1675){\makebox(0,0)[lb]{\smash{{\SetFigFont{6}{7.2}{\rmdefault}{\mddefault}{\updefault}{\color[rgb]{0,0,0}$v_2$}%
}}}}
\put(-1647,-1225){\makebox(0,0)[lb]{\smash{{\SetFigFont{6}{7.2}{\rmdefault}{\mddefault}{\updefault}{\color[rgb]{0,0,0}$v_3$}%
}}}}
\put(-297,-1225){\makebox(0,0)[lb]{\smash{{\SetFigFont{6}{7.2}{\rmdefault}{\mddefault}{\updefault}{\color[rgb]{0,0,0}$v_1$}%
}}}}
\end{picture}%
}}
\end{minipage}
\begin{minipage}[b]{0.20\textwidth}
\centering
\resizebox{0.9\width}{0.9\height}{
\begin{picture}(0,0)%
\includegraphics{operation_label.pdf}%
\end{picture}%
\setlength{\unitlength}{3522sp}%
\begin{picture}(1852,1394)(-1825,-1883)
\put(-1124,-849){\makebox(0,0)[lb]{\smash{{\SetFigFont{9}{10.8}{\rmdefault}{\mddefault}{\updefault}{\color[rgb]{0,0,0}$\mathcal{FD}_4$}%
}}}}
\put(-1124,-1636){\makebox(0,0)[lb]{\smash{{\SetFigFont{9}{10.8}{\rmdefault}{\mddefault}{\updefault}{\color[rgb]{0,0,0}$\mathcal{PI}_2$}%
}}}}
\end{picture}%
}
\end{minipage}
\begin{minipage}[b]{0.39\textwidth}
\centering
\subfloat[$G'$]{\label{fig:operation_2_cut_result}
\resizebox{0.85\width}{0.85\height}{
\begin{picture}(0,0)%
\includegraphics{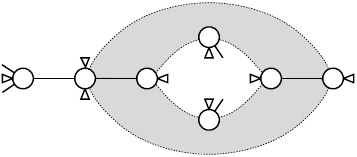}%
\end{picture}%
\setlength{\unitlength}{2901sp}%
\begin{picture}(3873,1692)(-1825,-2032)
\put(378,-775){\makebox(0,0)[lb]{\smash{{\SetFigFont{6}{7.2}{\rmdefault}{\mddefault}{\updefault}{\color[rgb]{0,0,0}$v_6$}%
}}}}
\put(378,-1675){\makebox(0,0)[lb]{\smash{{\SetFigFont{6}{7.2}{\rmdefault}{\mddefault}{\updefault}{\color[rgb]{0,0,0}$v_7$}%
}}}}
\put(1053,-1225){\makebox(0,0)[lb]{\smash{{\SetFigFont{6}{7.2}{\rmdefault}{\mddefault}{\updefault}{\color[rgb]{0,0,0}$v_8$}%
}}}}
\put(1728,-1225){\makebox(0,0)[lb]{\smash{{\SetFigFont{6}{7.2}{\rmdefault}{\mddefault}{\updefault}{\color[rgb]{0,0,0}$v_9$}%
}}}}
\put(-972,-1225){\makebox(0,0)[lb]{\smash{{\SetFigFont{6}{7.2}{\rmdefault}{\mddefault}{\updefault}{\color[rgb]{0,0,0}$v_5$}%
}}}}
\put(-1647,-1225){\makebox(0,0)[lb]{\smash{{\SetFigFont{6}{7.2}{\rmdefault}{\mddefault}{\updefault}{\color[rgb]{0,0,0}$v_3$}%
}}}}
\put(-297,-1225){\makebox(0,0)[lb]{\smash{{\SetFigFont{6}{7.2}{\rmdefault}{\mddefault}{\updefault}{\color[rgb]{0,0,0}$v_1$}%
}}}}
\end{picture}%
}}
\end{minipage}
\end{center}
\caption{Showing that $\mathcal{FD}_4(G)$ has a simple dual}
\label{fig:2_cut_operation}
\end{figure}
\end{proof}
\end{lem}

\begin{theo}\label{theo:type_1_grnbergian_4_face_deflatable}
Any Type 1 Grinbergian hypohamiltonian graph is\/ $4$-face deflatable. More precisely, any Type 1 Grinbergian hypohamiltonian graph of order\/ $n$ is in\/ $\mathcal{H}^4_0(n)\cup\mathcal{H}^4_1(n)$.

\begin{proof}
Let $G$ be a Type 1 Grinbergian hypohamiltonian graph with $n$ vertices. By Theorem~\ref{theo:grinbergian_hypohamiltonian_type} the exceptional face belongs to $\mathcal{P}_1$ so its size is 4 or it is larger. If the exceptional face is a 4-face, then by Lemma~\ref{lem:grinberg_type_1_degree_4} the 4-face has two non-adjacent 4-valent vertices. So we can apply $\mathcal{FD}_4$ to obtain a graph $G'$ which has no face of size less than 5. So $\delta(G'^*)\geq 5$ and $G'^*$ is a simple plane graph by Lemma~\ref{lem:fs_simple_dual}. Thus $G'^*\in \mathcal{D}_5$ and as a result of the definition of $\mathcal{M}^4_f$, $G'^{**}=G'\in \mathcal{M}^4_0(n-1)$. Furthermore, $G\in \mathcal{M}^4_1(n)$ because $G\in\mathcal{PI}_2(G')$ and as $G$ is hypohamiltonian, $G\in \mathcal{H}^4_1(n)$.

But if the exceptional face is not a 4-face, then by the fact that it is 3-connected and simple, $G^*$ is simple as well and as the minimum face size of $G$ is 5, $\delta(G^*)\geq 5$ which means $G\in\mathcal{M}^4_0(n)$ and so $G\in\mathcal{H}^4_0(n)$.
\end{proof}
\end{theo}

\begin{cor}
The smallest Type 1 Grinbergian hypohamiltonian graph has\/ $42$ vertices and there are exactly\/ $7$ of them on\/ $42$ vertices.

\begin{proof}
By Theorem~\ref{theo:type_1_grnbergian_4_face_deflatable} any Type 1 Grinbergian graph belongs to $\mathcal{H}^4_0(n)\cup\mathcal{H}^4_1(n)$ but according to the results presented in the paragraph preceding Theorem~\ref{theo:no_hypo_less_than_40}, we have $\mathcal{H}^4_0(n)\cup\mathcal{H}^4_1(n)=\emptyset$ for all $n < 42$. So there is no such graph of order less than 42. On the other hand, we have $\mathcal{H}^4_0(42)=\emptyset$ and $|\mathcal{H}^4_1(42)|=7$ which completes the proof.
\end{proof}
\end{cor}

\section{Results}\label{sec:corollaries}
\noindent We present one of the planar hypohamiltonian graphs of order 40 in Figure~\ref{fig:40-01}, and the other 24 in Figure~\ref{fig:40-all}.

\begin{figure}[htbp]
\begin{center}
\includegraphics[height=3cm]{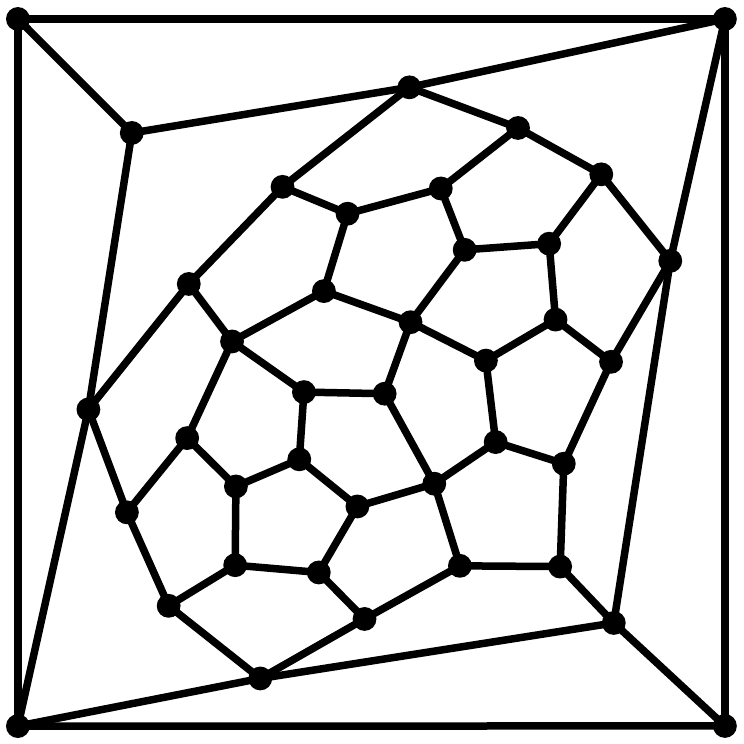}\\
\end{center}
\caption{A planar hypohamiltonian graph on 40 vertices}
\label{fig:40-01}
\end{figure}

\begin{theo}\label{theo:hypo_40}
The graph shown in Figure~\ref{fig:40-01} is hypohamiltonian.
\end{theo}

\begin{proof}
We first show that the graph is non-Hamiltonian. Assume to the contrary that the graph contains a Hamiltonian cycle, which must then satisfy Grinberg's Theorem. The graph in Figure~\ref{fig:40-01} contains five 4-faces and 22 5-faces. Then $$\sum_{i \ge 3} (i-2)(f_i - f'_i) \equiv 2(f_4-f'_4) \equiv 0 \!\!\!\pmod 3,$$ where $f_4+f'_4 = 5$. So $f'_4=1$ and $f_4=4$, or $f'_4=4$ and $f_4=1$. Let $Q$ be the 4-face on a different side from the other four.

Notice that an edge belongs to a Hamiltonian cycle if and only if the two faces it belongs to are on different sides of the cycle. Since the outer face of the embedding in Figure~\ref{fig:40-01} has edges in common with all other 4-faces and its edges cannot all be in a Hamiltonian cycle, that face cannot be $Q$.

If $Q$ is any of the other 4-faces, then the only edge of the outer face in the embedding in Figure~\ref{fig:40-01} that belongs to a Hamiltonian cycle is the edge belonging to $Q$ and the outer face. The two vertices of the outer face that are not endpoints of that edge have degrees 3 and 4, and we arrive at a contradiction as we know that two of the edges incident to the vertex with degree 3 are not part of the Hamiltonian cycle. Thus, the graph is non-Hamiltonian.

Finally, for each vertex it is routine to exhibit a
cycle of length 39 that avoids it.
\end{proof}

We now employ an operation introduced by Thomassen \cite{T5} for producing infinite sequences of hypohamiltonian graphs. Let $G$ be a graph containing a $4$-cycle $v_1 v_2 v_3 v_4 = C$. We denote by ${\rm Th}(G^C)$ the graph obtained from $G$ by deleting the edges $v_1 v_2$, $v_3 v_4$ and adding a new 4-cycle $v'_1 v'_2 v'_3 v'_4$ and the edges $v_i v'_i$, $1 \le i \le 4$. Araya and Wiener \cite{WA} note that a result in \cite{T5} generalizes as follows, with the same proof.



\begin{lem}\label{lem:TH_G_C_cubic_hypohamiltonian}
Let\/ $G$ be a planar hypohamiltonian graph containing a\/ $4$-face bounded by a cycle\/ $v_1 v_2 v_3 v_4 = C$ with cubic vertices. Then\/ ${\rm Th}(G^C)$ is also a planar hypohamiltonian graph.
\end{lem}

Araya and Wiener use this operation to show that planar hypohamiltonian graphs exist for every order greater than or equal to $76$. That result is improved further in the next theorem.

\begin{theo}\label{theo:3_hypohamiltonian_40_42}
There exist planar hypohamiltonian graphs of order\/ $n$ for every\/ $n \geq 42$.
\end{theo}

\begin{proof}
Figures \ref{fig:40-01}, \ref{fig:records}, and \ref{fig:records2} show plane hypohamiltonian graphs on 40, 42, 43, and 45 vertices, respectively. It can be checked that applying the Thomassen operation to the outer face of these plane graphs gives planar hypohamiltonian graphs with 44, 46, 47, and 49 vertices. By the construction, these graphs will have a 4-face bounded by a cycle with cubic vertices, so the theorem now follows from repeated application of the Thomassen operation and Lemma~\ref{lem:TH_G_C_cubic_hypohamiltonian}.
\end{proof}

\begin{figure}[htbp]
\begin{center}
\begin{minipage}[b]{0.45\linewidth}
\centering
\includegraphics[height=3cm]{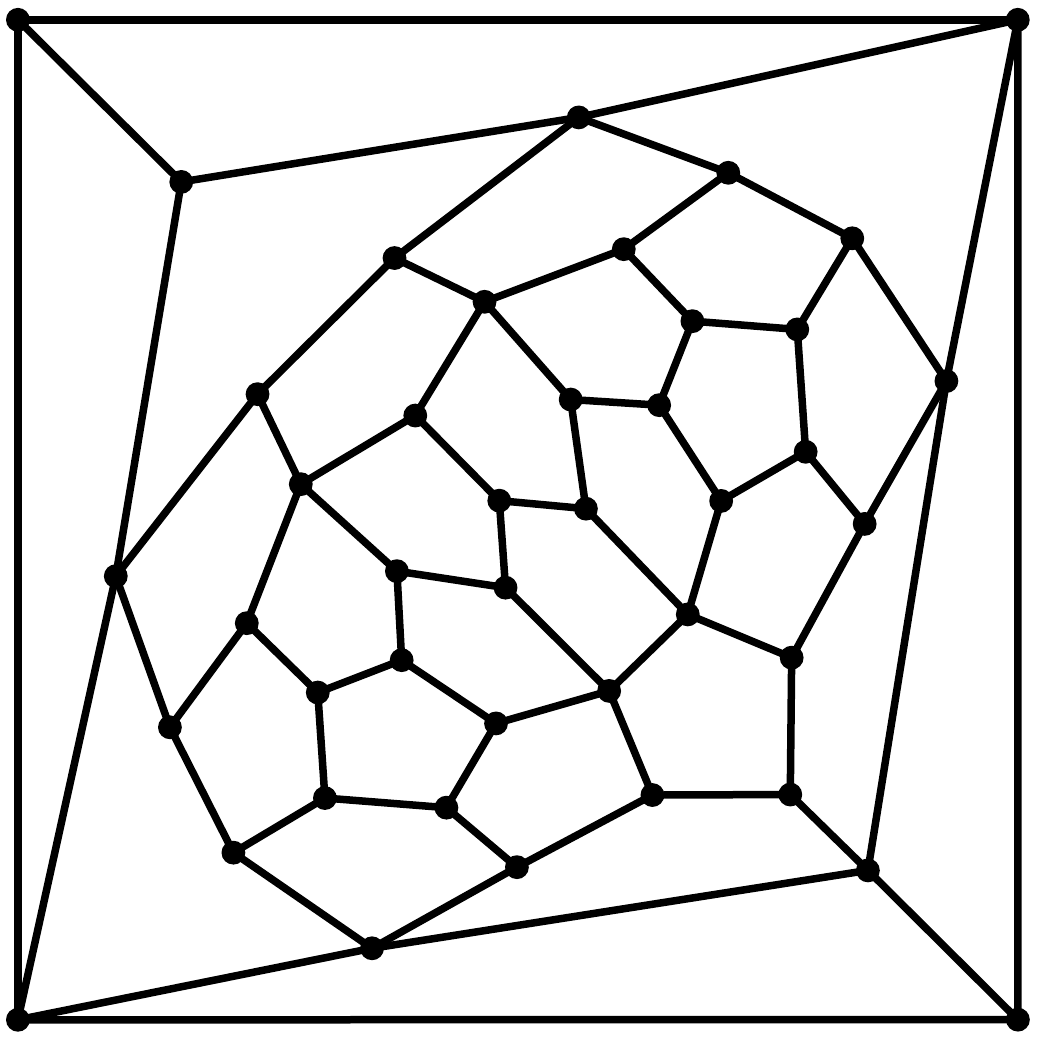}
\end{minipage}
\begin{minipage}[b]{0.45\linewidth}
\centering
\includegraphics[height=3cm]{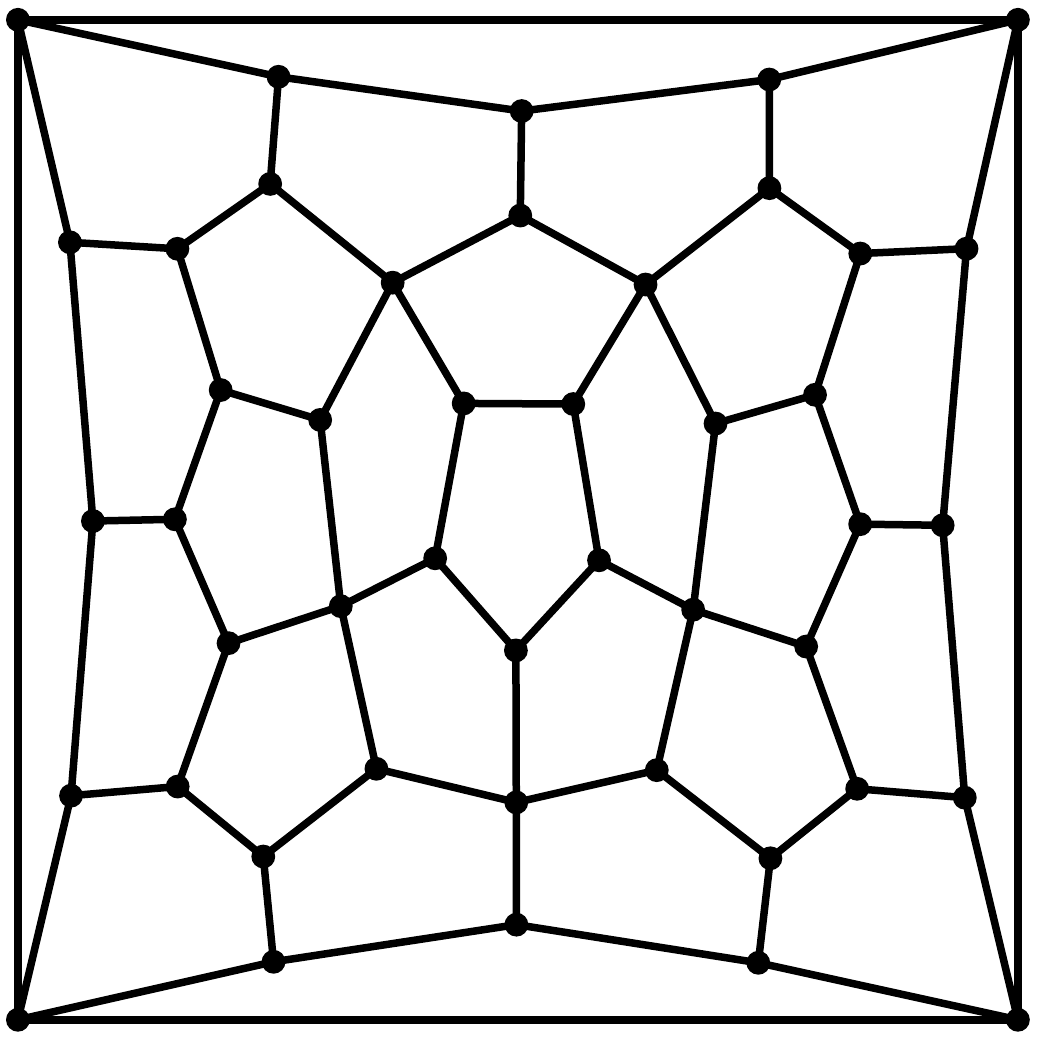}
\end{minipage}
\end{center}
\caption{Planar hypohamiltonian graphs of order 43 and 45, respectively}
\label{fig:records2}
\end{figure}

Whether there exists a planar hypohamiltonian graph on $41$ vertices remains an open question.

Araya and Wiener \cite{WA} further prove that there exist planar hypotraceable graphs on $162+4k$ vertices for every $k \ge 0$, and on $n$ vertices for every $n \ge 180$. To improve on that result, we make use of the following theorem, which is a slight modification of \cite[Lemma 3.1]{T1}.

\begin{theo}
\label{theo:trace}
If there are four planar hypohamiltonian graphs\/ $G_i = (V_i,E_i)$, $1 \leq i \leq 4$, each of which has a vertex of degree\/ $3$, then there is a planar hypotraceable graph of order\/ $|G_1|+|G_2|+|G_3|+|G_4|-6$.
\end{theo}

\begin{proof}
The result follows from the proof of~\cite[Lemma~1]{T1} and the fact that the construction used in that proof (which does not address planarity) can be carried out to obtain a planar graph when all graphs $G_i$ are planar.
\end{proof}

\begin{theo}\label{theo:hypotraceable_154_and_above_155}
There exist planar hypotraceable graphs on\/ $154$ vertices, and on\/ $n$ vertices for every\/ $n \geq 156$.
\end{theo}

\begin{proof}
All the graphs obtained in the proof of Theorem~\ref{theo:3_hypohamiltonian_40_42} have a vertex with degree 3. Consequently, Theorem~\ref{theo:trace} can be applied to those graphs to obtain planar hypotraceable graphs of order $n$ for $n = 40+40+40+40-6 = 154$ and for $n \geq 40+40+40+42-6 = 156$.
\end{proof}

The graphs considered in this work have girth 4. In fact, by the following theorem we know that any planar hypohamiltonian graphs improving on the results of the current work must have girth 3 or 4. Notice that, perhaps surprisingly, there exist planar hypohamiltonian graphs of girth 3~\cite{T2}, and that a planar hypohamiltonian graph can have girth at most 5, since such a graph has a simple dual, and the average degree of a simple plane graph is less than 6.

\begin{theo}
There are no planar hypohamiltonian graphs with girth\/ $5$ on fewer than\/ $45$ vertices, and there is exactly one such graph on\/ $45$ vertices.
\end{theo}

\begin{proof}
The program \emph{plantri} \cite{BM07} can be used to construct all planar graphs with a simple dual, girth 5, and up to 45 vertices. By checking these graphs, it turns out that only a single graph of order 45 is hypohamiltonian. That graph, which has an automorphism group of order 4, is shown in Figure~\ref{fig:45-P}.
\end{proof}

\begin{figure}[htbp]
\begin{center}
\includegraphics[height=3cm]{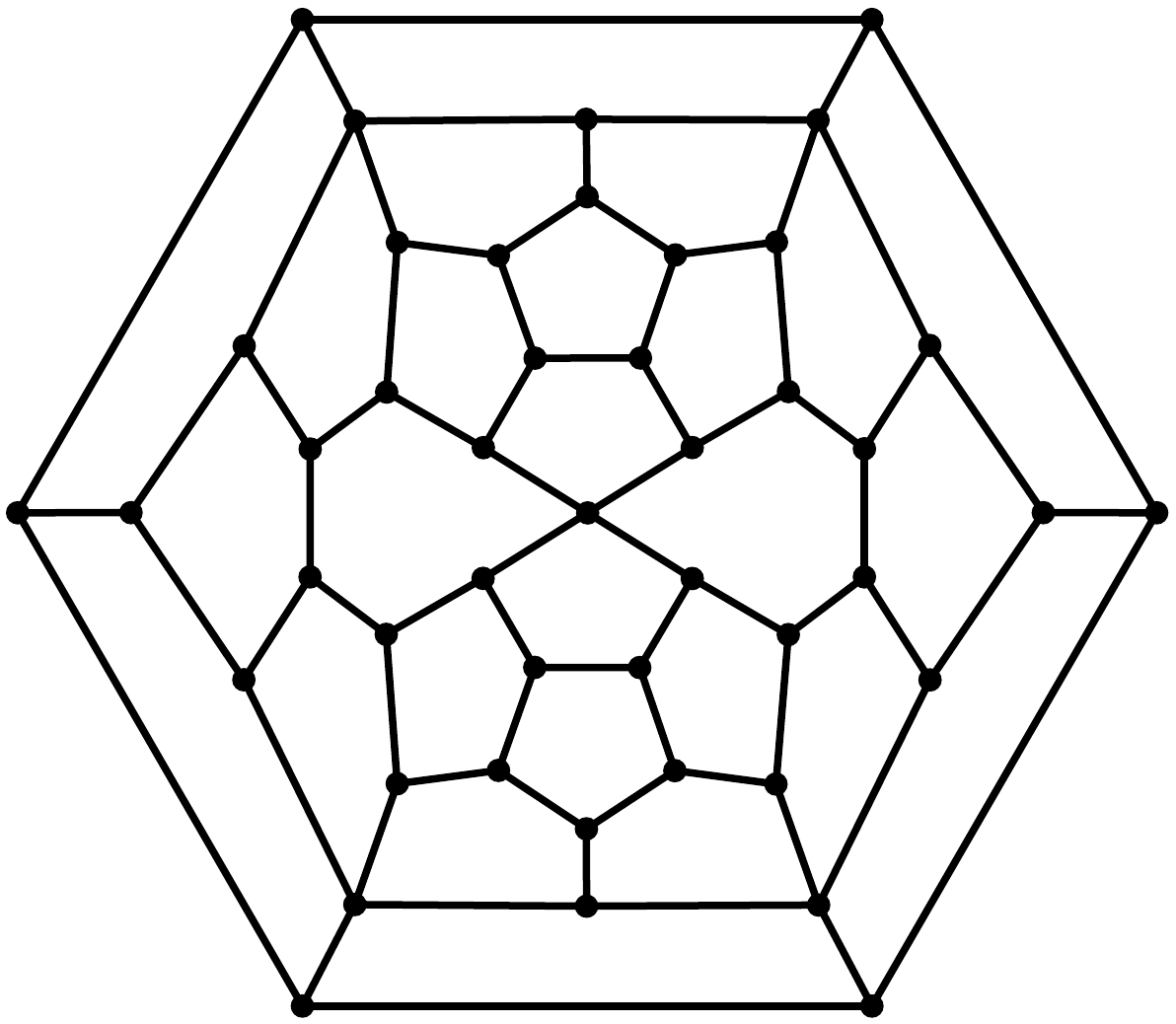}\\
\end{center}
\caption{A planar hypohamiltonian graph with girth 5 and 45 vertices}
\label{fig:45-P}
\end{figure}

Let $H$ be a cubic graph and $G$ be a graph containing a cubic vertex $w \in V(G)$. We say that we \emph{insert $G$ into $H$}, if we replace every vertex of $H$ with $G - w$ and connect the endpoints of edges in $H$ to the neighbours of $w$.

\begin{cor}
We have $$\overline{C^1_3}\leq 40, \quad \overline{C^2_3}\leq 2625, \quad \overline{P^1_3}\leq 156, \quad \mbox{ and } \quad \overline{P^2_3}\leq 10350.$$

\begin{proof}
The first of the four inequalities follows immediately from Theorem~\ref{theo:hypo_40}. In the following, let $G$ be the planar hypohamiltonian graph from Figure~\ref{fig:40-01}.

For the second inequality, insert $G$ into Thomassen's hypohamiltonian graph $G_0$ from \cite[p.~38]{T5}. This means that each vertex of $G_0$ is replaced by $G$ minus some vertex of degree 3. Since every pair of edges in $G_0$ is missed by a longest cycle \cite[p.~156]{SZ}, in the resulting graph $G'$ any pair of vertices is missed by a longest cycle. This property is not lost if all edges originally belonging to $G_0$ are contracted.

In order to prove the third inequality, insert $G$ into $K_4$. We obtain a graph in which every vertex is avoided by a path of maximal length.

For the last inequality, consider the graph $G_0$ from the second paragraph of this proof and insert $G_0$ into $K_4$ to obtain $H$. Now insert $G$ into $H$. Finally, contract all edges which originally belonged to $H$.
\end{proof}
\end{cor}

For a more detailed discussion of the above corollary and proof, please consult \cite{Z3} and the recent survey \cite{SZZ}.

\newcommand{\ahgh}{16mm}
\begin{figure}[htbp]
\begin{center}
\includegraphics[height=\ahgh]{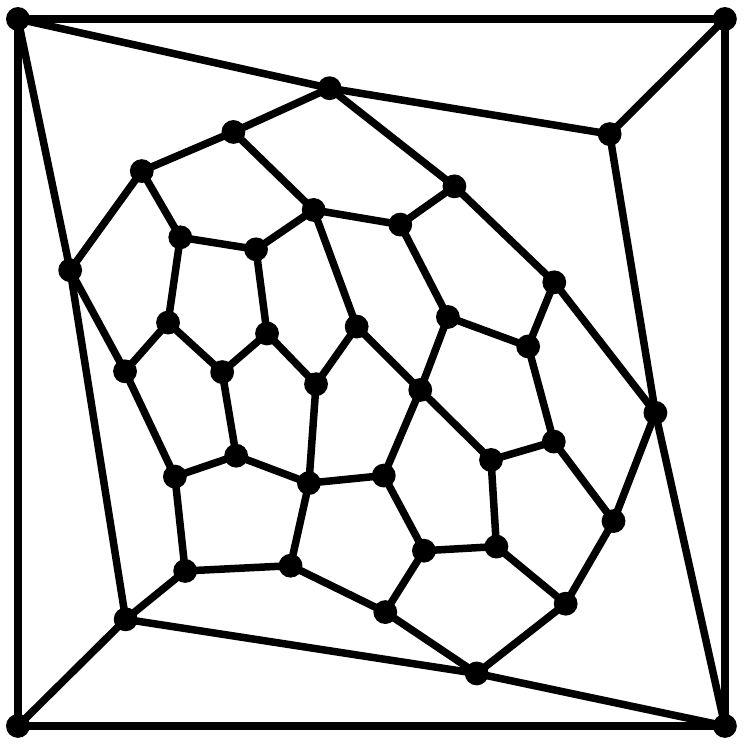}
\includegraphics[height=\ahgh]{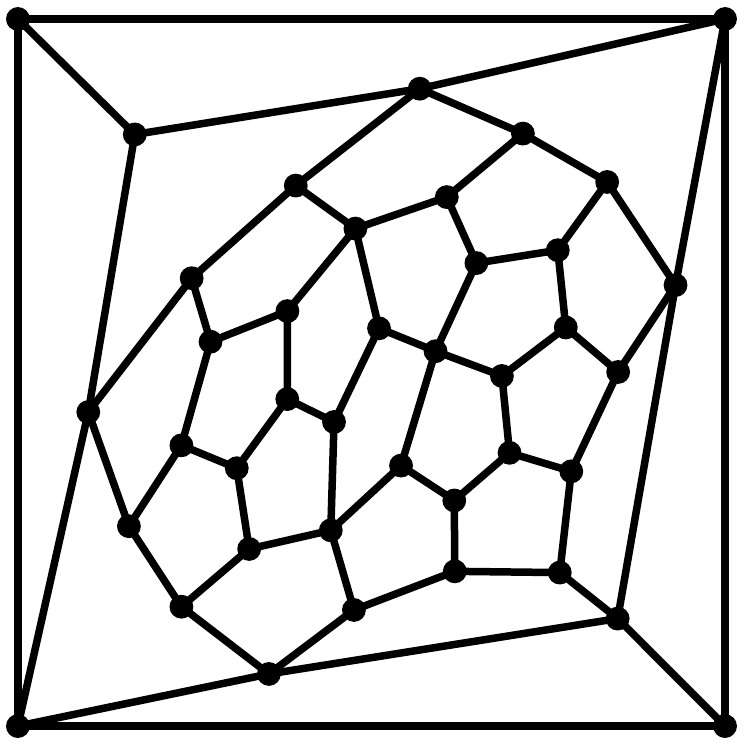}
\includegraphics[height=\ahgh]{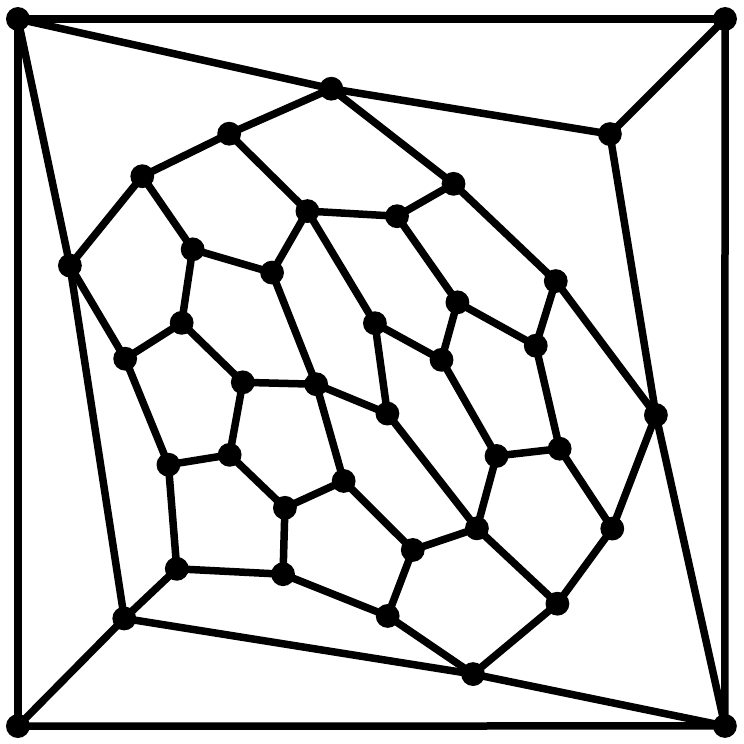}
\includegraphics[height=\ahgh]{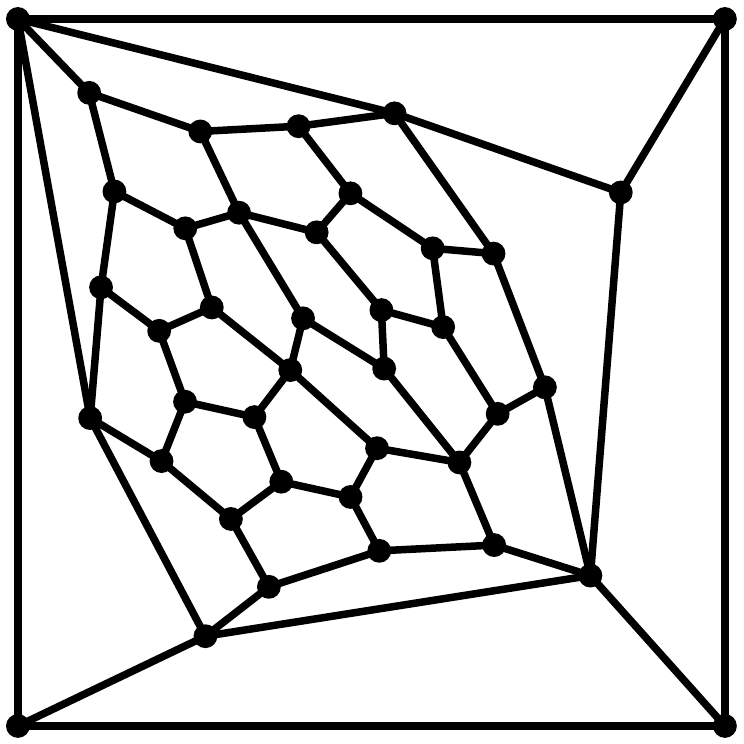}
\includegraphics[height=\ahgh]{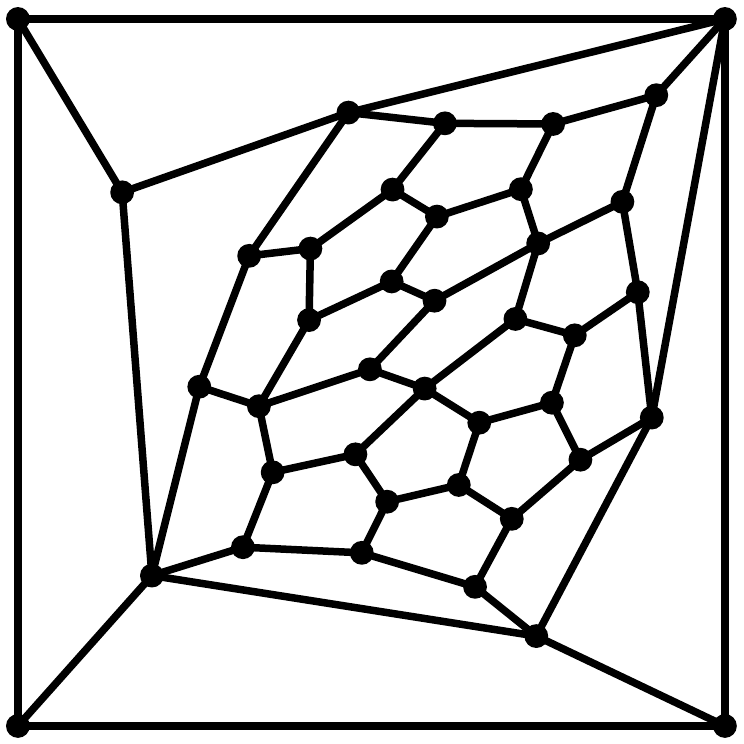}
\includegraphics[height=\ahgh]{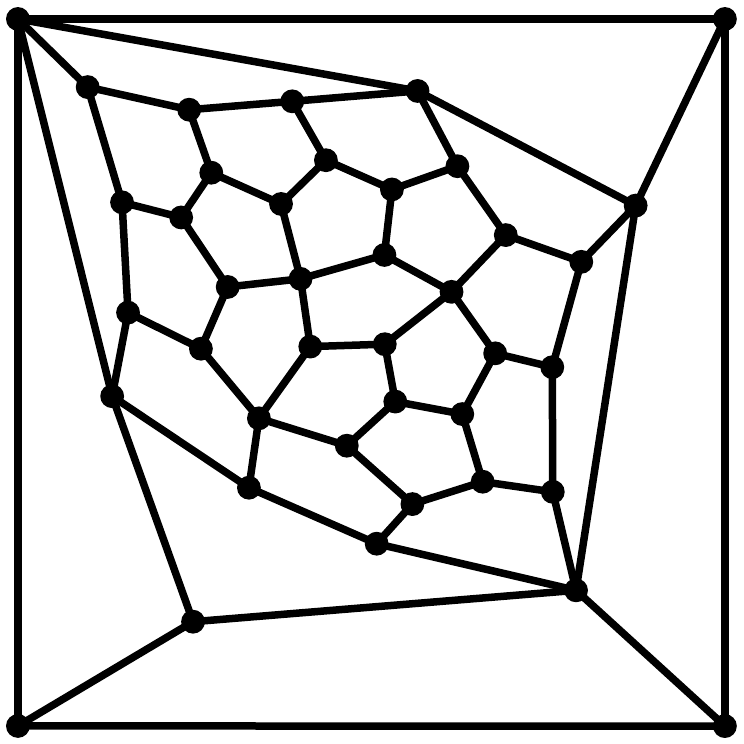}
\includegraphics[height=\ahgh]{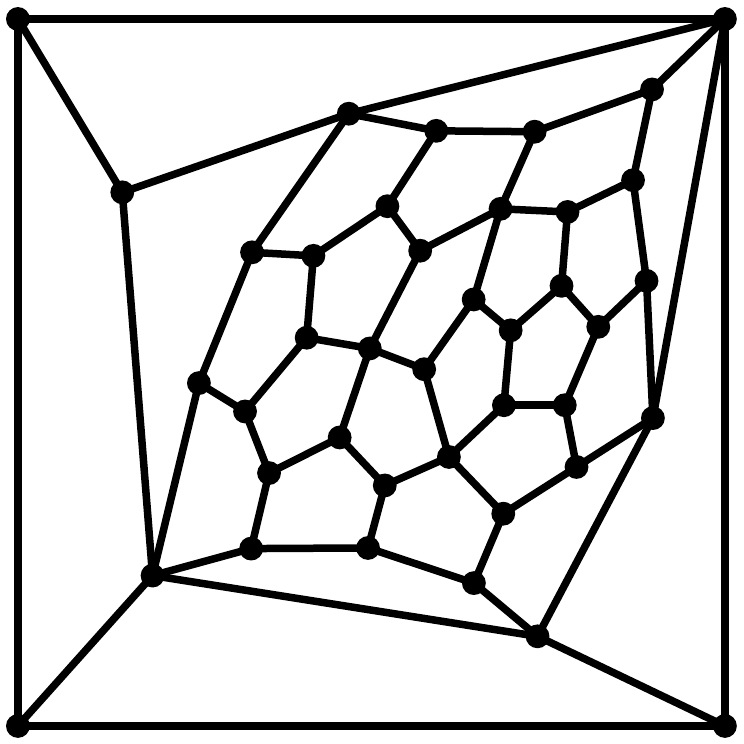}
\includegraphics[height=\ahgh]{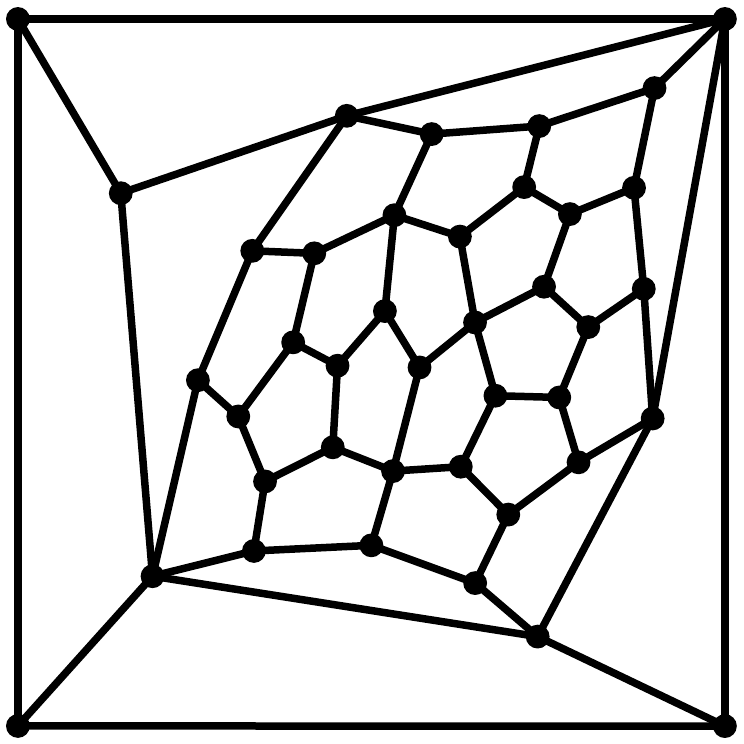}
\\\vspace{1mm}
\includegraphics[height=\ahgh]{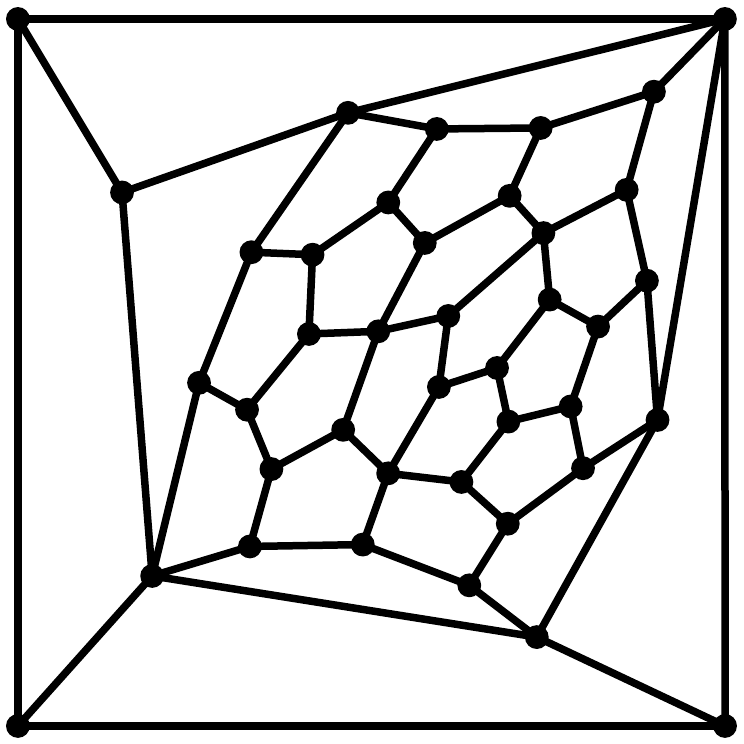}
\includegraphics[height=\ahgh]{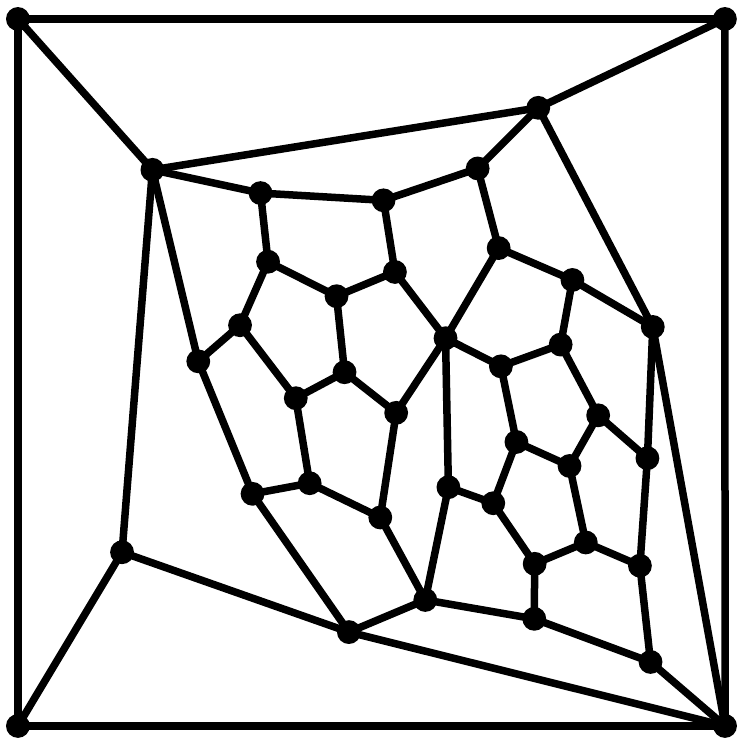}
\includegraphics[height=\ahgh]{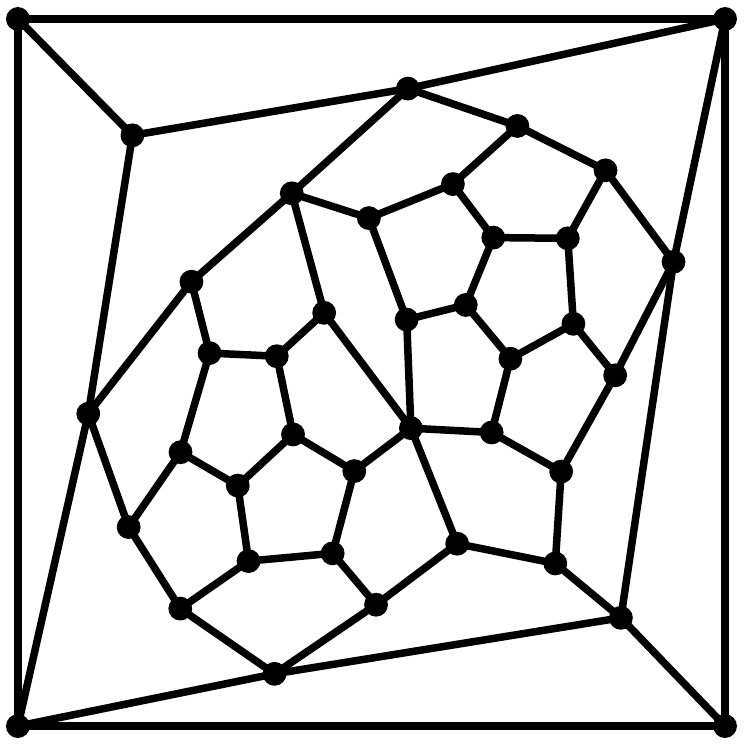}
\includegraphics[height=\ahgh]{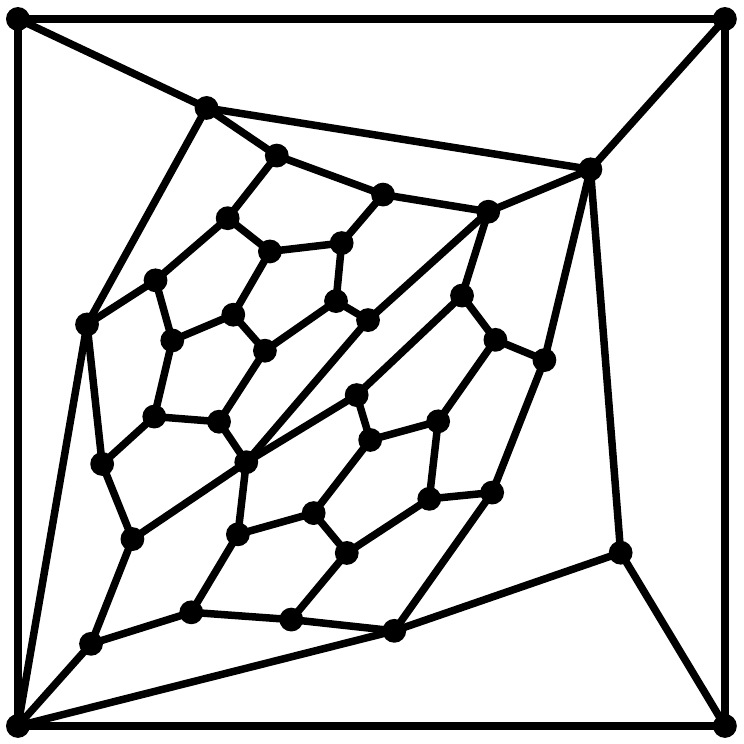}
\includegraphics[height=\ahgh]{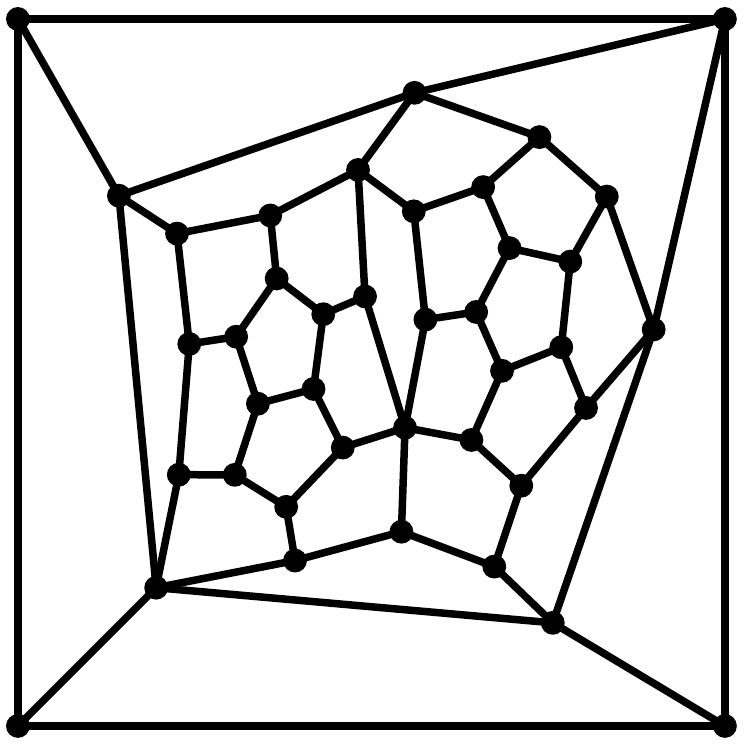}
\includegraphics[height=\ahgh]{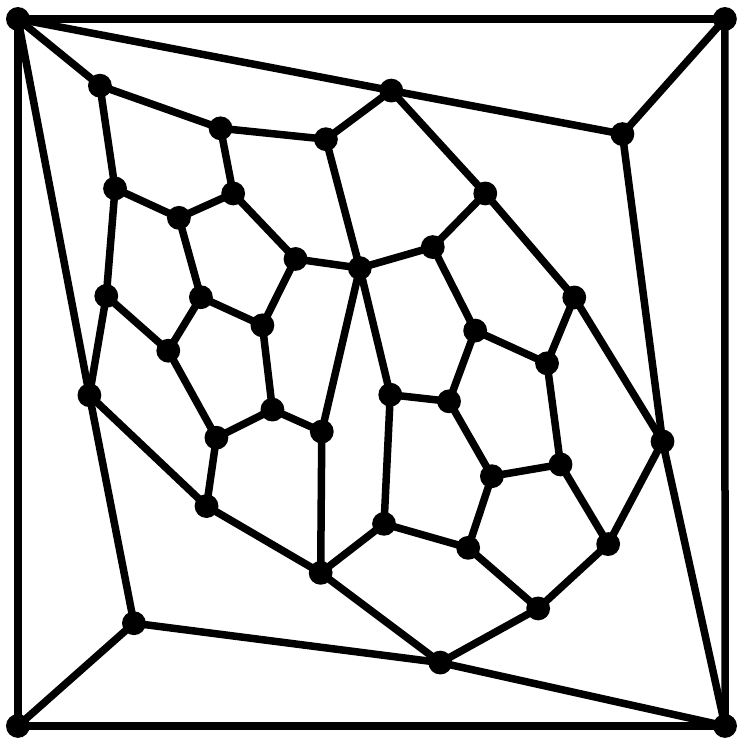}
\includegraphics[height=\ahgh]{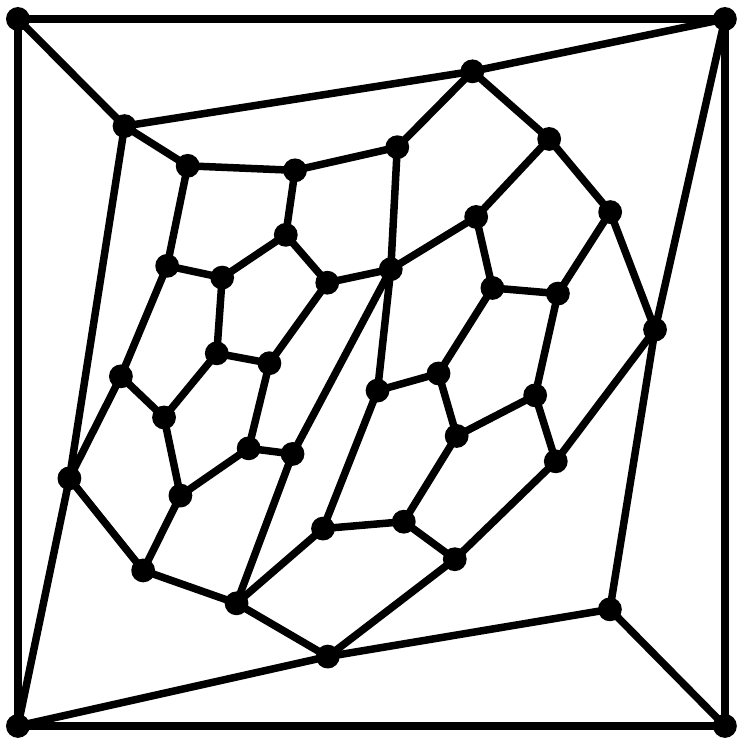}
\includegraphics[height=\ahgh]{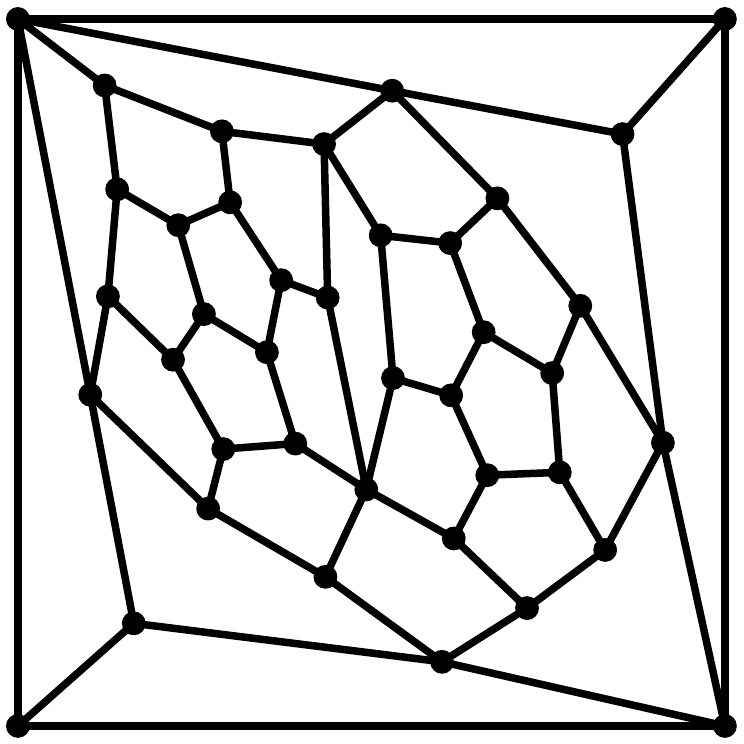}
\\\vspace{1mm}
\includegraphics[height=\ahgh]{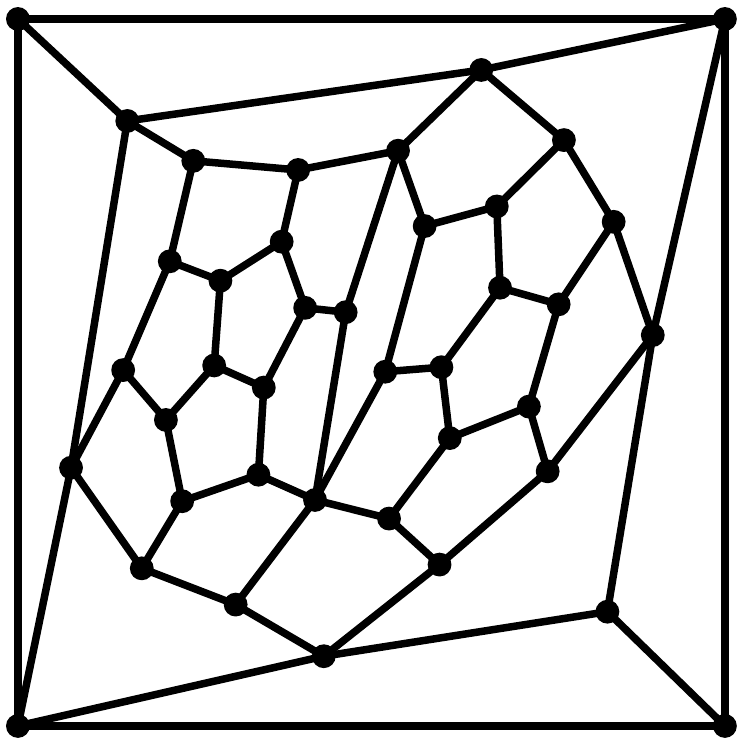}
\includegraphics[height=\ahgh]{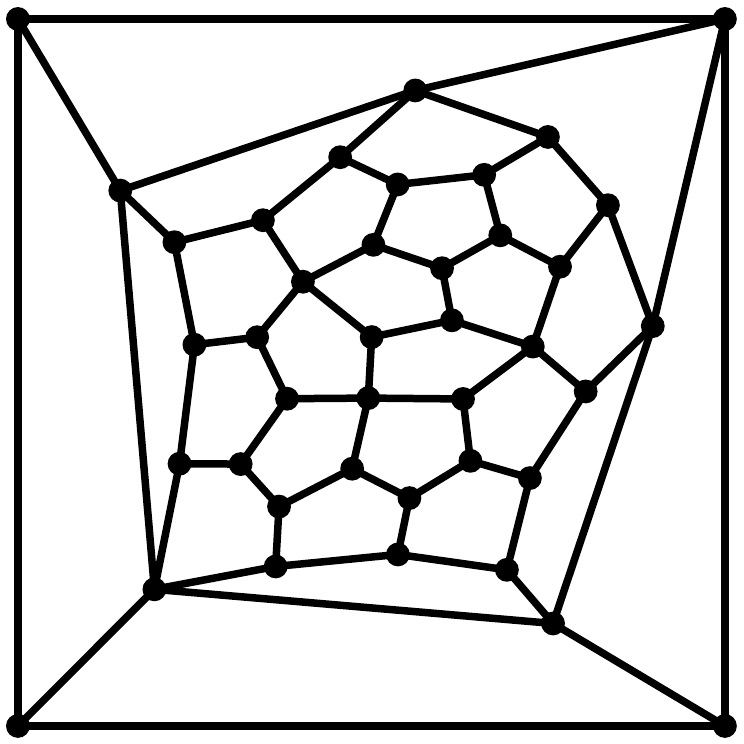}
\includegraphics[height=\ahgh]{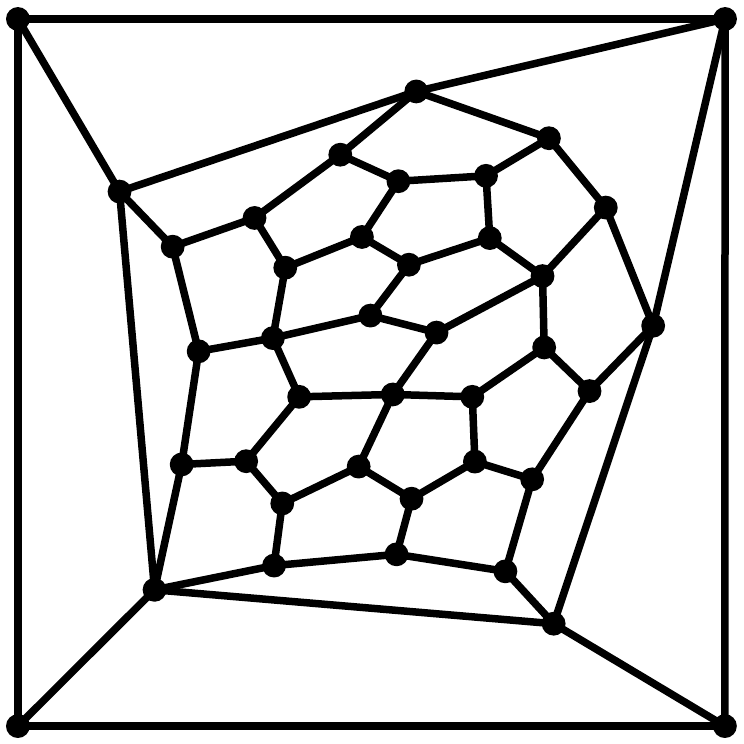}
\includegraphics[height=\ahgh]{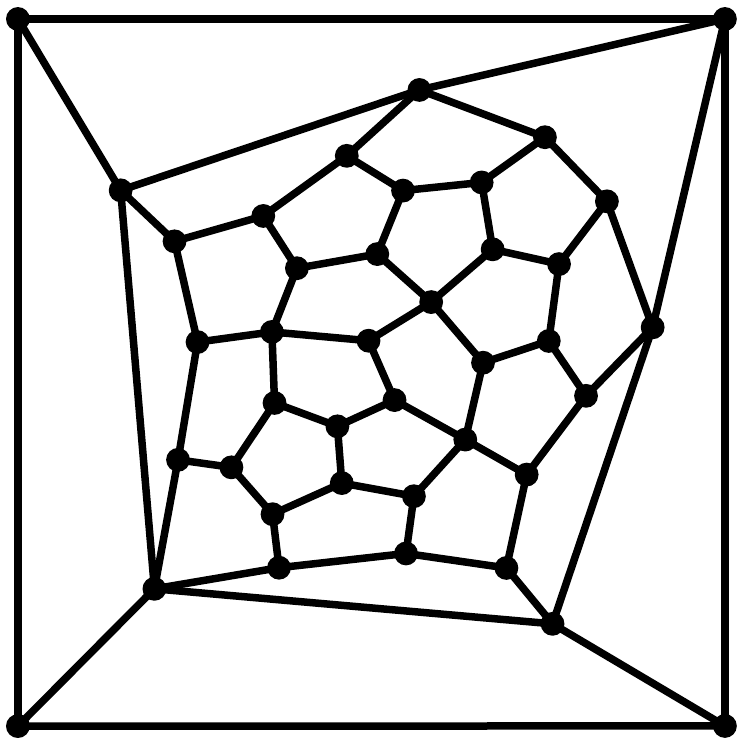}
\includegraphics[height=\ahgh]{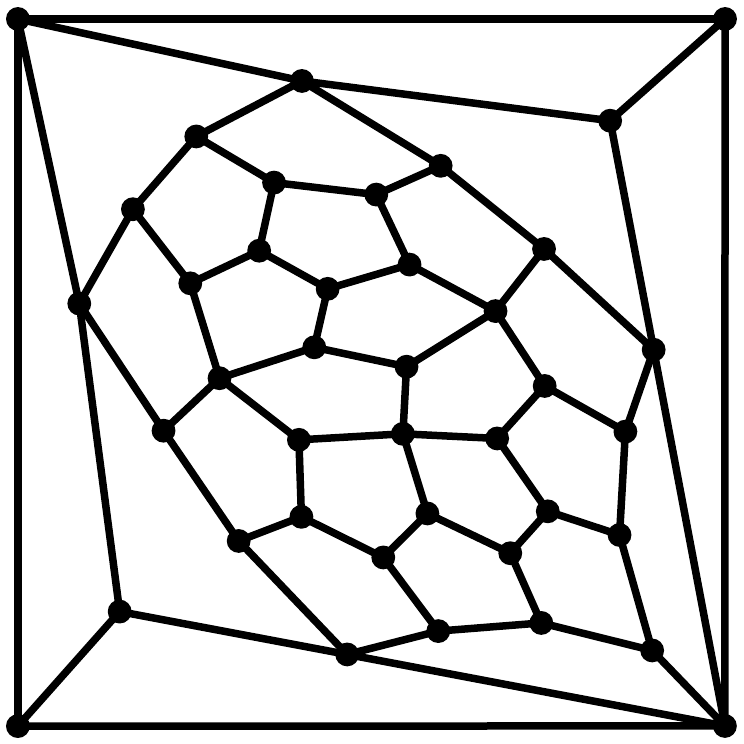}
\includegraphics[height=\ahgh]{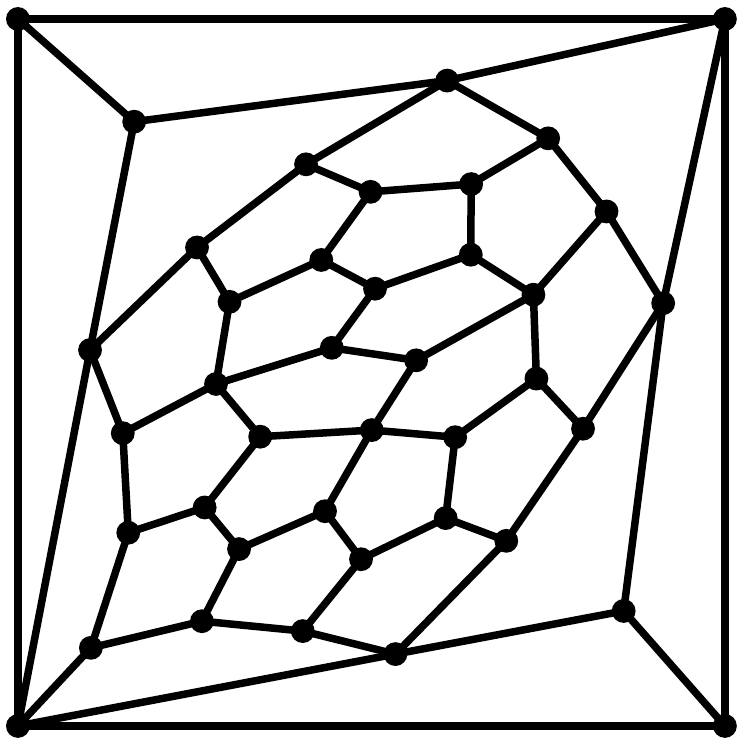}
\includegraphics[height=\ahgh]{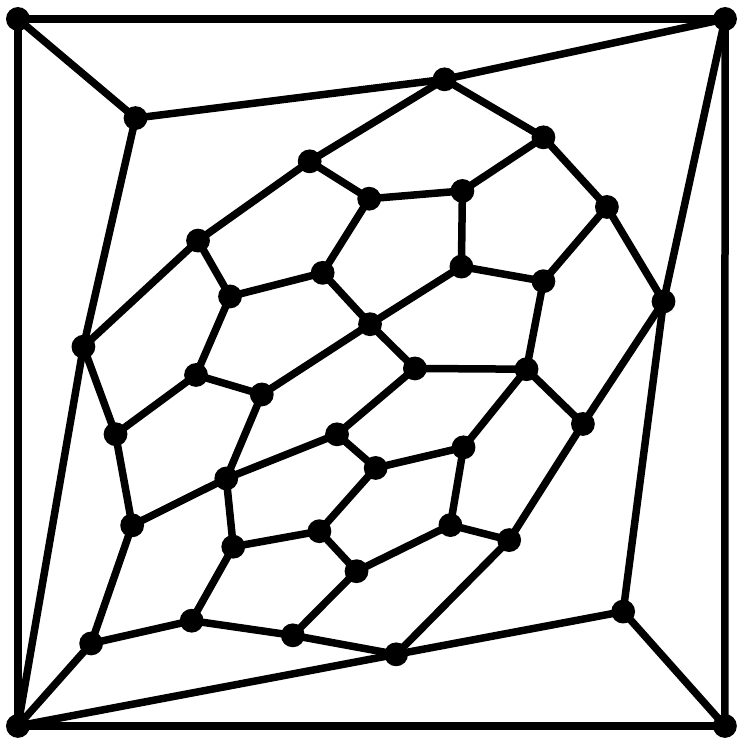}
\includegraphics[height=\ahgh]{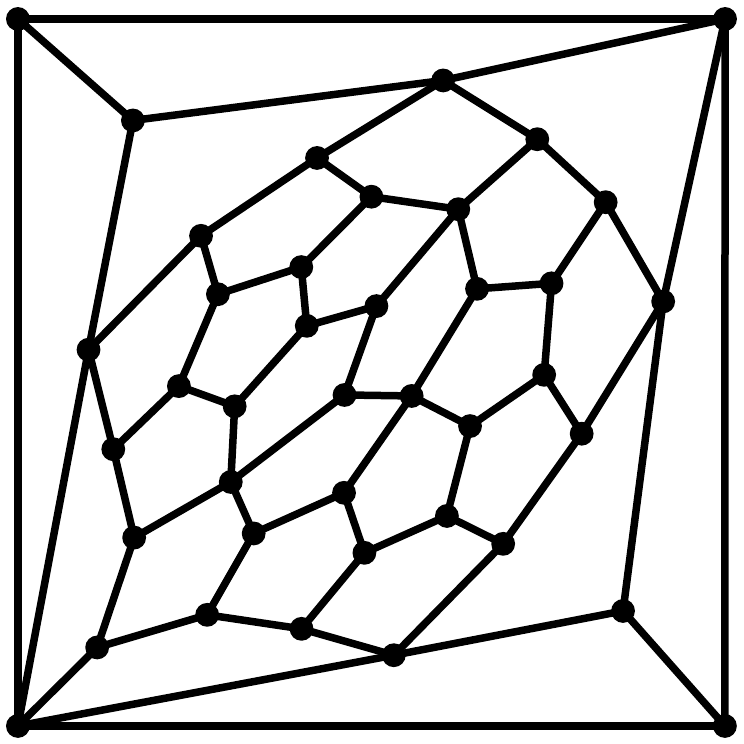}
\end{center}
\caption{The rest of the hypohamiltonian graphs on 40 vertices}
\label{fig:40-all}
\end{figure}

\section{Conclusions}\label{sec:con}
\noindent Despite the new planar hypohamiltonian graphs discovered in the current work, there is still a wide gap between the order of the smallest known graphs and the best lower bound known for the order of the smallest such graphs, which is 18 \cite{AMW}. One explanation for this gap is the fact that no extensive computer search has been carried out to increase the lower bound.

It is encouraging though that the order of the smallest known planar hypohamiltonian graph continues to decrease. It is very difficult to conjecture anything about the smallest possible order, and the possible extremality of the graphs discovered here. It would be somewhat surprising if no extremal graphs would have nontrivial automorphisms (indeed, the graphs of order 40 discovered in the current work have no nontrivial automorphisms). An exhaustive study of graphs with prescribed automorphisms might lead to the discovery of new, smaller graphs.

The smallest known \emph{cubic} planar hypohamiltonian graph has 70 vertices~\cite{AW}. We can hope that the current work inspires further progress in that problem, too.

\section*{Acknowledgements}
\noindent The first two authors were supported by the Australian Research Council. The work of the third author was supported in part by the Academy of Finland under the Grant No.\ 132122; the work of the fourth author was supported by the same grant, by the GETA Graduate School, and by the Nokia Foundation.





\bibliographystyle{elsart-num-sort}
\bibliography{Hypoham40}







\end{document}